%% file: paper.tex
\documentclass[]{siamart190516}
\RequirePackage{amssymb,amsmath,bm,mathtools,amsfonts}
\RequirePackage{graphicx,xcolor,subfig}
\usepackage{epstopdf}
\graphicspath{{./fig/}}

\usepackage{relsize}
\usepackage{fixmath} % for \mathbold
\usepackage{scalerel}
\usepackage{eufrak}
\usepackage{ulem}
\usepackage{yfonts}
\newlength\bshft
\def\fakebold#1{\ThisStyle{\ooalign{$\SavedStyle#1$\cr%
\kern-\bshft$\SavedStyle#1$\cr%
\kern\bshft$\SavedStyle#1$}}}

\newcommand{\norm}[1]{\left\Vert {#1} \right\Vert}

\newcommand{\abs}[1]{\left\vert {#1} \right\vert}

\newcommand{\defeq}{\vcentcolon=}

\DeclareMathOperator{\diver}{{\rm div}}
\DeclareMathOperator{\curl}{\mathbf{curl}}

\renewcommand{\gg}{\boldsymbol{g}}
\newcommand{\xx}{\boldsymbol{x}}

\newcommand{\mm}{\boldsymbol{m}}
\renewcommand{\AA}{\boldsymbol{A}}

\newcommand{\vv}{\boldsymbol{v}}

\newcommand{\ww}{\boldsymbol{w}}

\newcommand{\nn}{\boldsymbol{n}}

\newcommand{\PiV}{{\bm \Pi}_k^0}
\newcommand{\R}{\mathbb{R}}

\newcommand{\err}{\texttt{err}}
\newcommand{\errApp}{\texttt{errApp}}
\newcommand{\errGeo}{\texttt{errGeo}}
\renewcommand{\P}{\mathbb{P}}

\newcommand{\ds}{~\text{d}\mathfrak{F}}
\newcommand{\dF}{~\text{d}F}
\newcommand{\dO}{~\text{d}\Omega}
\newcommand{\dP}{~\text{d}P}
\newcommand{\dt}{~\text{d}t}
\newcommand{\F}{\mathbf{F}}

\newsiamthm{prop}{Proposition}%[section]
\newsiamthm{remark}{Remark}%[section]
\newsiamthm{problem}{Problem}
\newsiamthm{dof}{Degrees of freedom}
\newsiamthm{example}{Esempio}
\newsiamthm{property}{Property}
\newsiamthm{assumption}{Assumption}
\usepackage[all,cmtip]{xy}

\usepackage{todonotes}

\newcommand{\as}[1]{#1}
\newcommand{\af}[1]{#1}
\renewcommand{\sout}[1]{}
\renewcommand{\todo}[2][1]{}

\title{Bend 3d Mixed Virtual Element Method \as{for elliptic problems}}

\author{Franco Dassi\thanks{Dipartimento di Matematica e Applicazioni, Universit\`a degli Studi di Milano Bicocca, Via Roberto Cozzi 55 - 20125 Milano, Italy (\email{franco.dassi@unimib.it}).}
\and Alessio Fumagalli\thanks{MOX - Dipartimento di Matematica ``F. Brioschi'',
Politecnico di Milano, via Bonardi 9, 20133 Milan, Italy
(\email{alessio.fumagalli@polimi.it}, \email{anna.scotti@polimi.it}).}
\and Anna Scotti\footnotemark[2]
\and Giuseppe Vacca\thanks{Dipartimento di Matematica, Universit\`a degli Studi
di Bari, Piazza Umberto I - 70121 Bari, Italy
(\email{giuseppe.vacca@uniba.it})}
}

\allowdisplaybreaks[1]
\begin{document}

\maketitle

%\begin{frontmatter}
%\author[1]{Franco Dassi\corref{cor1}\fnref{fn1}}\ead{franco.dassi@unimib.it}%
%\author[2]{Alessio Fumagalli\fnref{fn1}}\ead{alessio.fumagalli@polimi.it}%
%% \author[2]{Davide Losapio\fnref{fn1}}\ead{davide.losapio@polimi.it}%
%% \author[3]{Stefano Scial\`o\fnref{fn1}}\ead{stefano.scialo@polito.it}%
%\author[2]{Anna Scotti\fnref{fn1}}\ead{anna.scotti@polimi.it}%
%\author[1]{Giuseppe Vacca\fnref{fn1}}\ead{giuseppe.vacca@unimib.it}%

%\fntext[fn1]{Member of INdAM-GNCS research group}
%\address[1]{Dipartimento di Matematica e Applicazioni, Universit\`a degli Studi di Milano Bicocca, Via Roberto Cozzi 55 - 20125 Milano, Italy}
%\address[2]{MOX - Dipartimento di Matematica ``F. Brioschi'', Politecnico di Milano, via Bonardi 9, 20133 Milan, Italy}
%%\address[3]{MOX - Dipartimento di Matematica ``F. Brioschi'', Politecnico di Milano, via Bonardi 9, 20133 Milan, Italy}
%% \address[3]{Dipartimento di Scienze Matematiche, Politecnico di Torino, Corso Duca degli Abruzzi 24, 10129 Torino, Italy}
%%\address[5]{MOX - Dipartimento di Matematica ``F. Brioschi'', Politecnico di Milano, via Bonardi 9, 20133 Milan, Italy}
%%\address[6]{Dipartimento di Matematica e Applicazioni, Universit\`a degli Studi di Milano Bicocca, Via Roberto Cozzi 55 - 20125 Milano, Italy}
%
\input{abstract}

%\end{frontmatter}

\input{intro}
\input{notazioni}
\input{problem}
\input{vemSpaces}
\input{stability}

\input{integration}

\input{numExe}

\input{conc}

\bibliographystyle{elsarticle-num}
\bibliography{biblio_no_url}

\end{document}

%% file: abstract.tex
\begin{abstract}
    In this study, we \sout{consider a} \as{propose a} virtual
    element scheme to solve the Darcy problem in three physical dimensions.
    The main novelty, here proposed, is that curved elements are naturally
    handled without any degradation of the solution accuracy. In fact, in
    presence of curved boundaries, or internal interfaces, the geometrical error
    introduced \as{by planar approximations} may dominate the convergence rate limiting the \sout{usage} \as{benefit} of
    high-order approximations. 
    %We consider the Darcy problem written in mixed
    %form, and thus the scheme, to directly obtain fluxes between elements that are %accurate, locally
    %mass conservative, and can be used in other physical processes without any
    %post-process.
    \as{We consider the Darcy problem in its mixed form to directly obtain, with our numerical scheme, accurate and mass conservative fluxes without any postprocessing.}
    An important step to derive this new scheme
    is the actual computation of polynomials over curved polyhedrons, here
    presented and discussed. Finally, we show the theoretical analysis of the scheme as
    well as several numerical examples to support our findings.
\end{abstract}

%% file: intro.tex
\section{Introduction}\label{sec:intro}

Fluid flow in porous media is a broad research area that involves scientists as
well as engineers from both private and public sector \as{due to its relevance in}
\sout{Very successful examples are related to}
energy management, for instance hydrocarbon extraction, long term  CO$_2$ sequestration, geothermal extraction and storing, \as{as well as in several applications in industry and life science.}
Moreover, an interesting and \sout{actual} \as{current} application is the prevention of groundwater pollution due to
human activities like landfills leakage, nuclear waste disposal, and remediation
of polluted industrial sites~, see e.g.~\cite{bear1993flow,MacMinn2010}.

\as{In subsurface flow simulations, and in particular in oil field modeling,}   \sout{Concerning the oil flow in the ground,}
the so-called {corner-point grid} is one of the most common ways \as{to discretize the domain.}
Corner point grids are composed by hexahedrons whose \as{facets} \sout{boundaries} are curved, more specifically they are bilinear functions~\cite{Aarnes2008}.
To compute geometrical properties on such meshes, the
standard approach consists in discretizing these cells by means of suitable sub-grids. However, in the presence of particular configurations,
such as pinched or highly distorted elements,
the geometrical error introduced by \as{this approximation based on a sub-tetrahedrizatoin,}  may pollute the accuracy of the whole solution.
\sout{Wells are crucial in the management of the reservoir since
the outflow of fluids in normally enforced by injecting at high pressure. (toglierei questa frase, blocca il flusso del discorso)}
In the context of the Finite Element Method (FEM),
an accurate description of the curved geometry is a key point
to preserve the quality of the solution itself especially for high order approximations.

In this paper, we propose a new numerical strategy to overcome this issue, i.e.,
we propose a method that avoids any geometrical error and
incorporates the presence of cells with curved faces {without} \as{the need for} sub-grids.
Such strategy is based on the Virtual Element Method (VEM) and
it is inspired by the seminal works~\cite{BeiraodaVeiga2019,Dassi2020,Dassi2021a}.
The VEM is an extension of the FEM to deal with polygonal/polyhedral grids.
Indeed, it follows mostly the classical finite element framework
but the key point is that the basis functions are left unknown and never computed~\cite{BeiraodaVeiga2013a},
It was successfully exploited for both 2d and 3d problemsc see, e.g.,~\cite{BeiraodaVeiga2014a,Brezzi2014,BeiraodaVeiga2017,BeiraodaVeiga2018,BeiraodaVeiga2020}
The possibility to deal with elements characterized by very general shape (also non-covex)
is appealing for real life problems.
In particular VEM have already been extensively employed in reservoir simulations,
where meshes are characterized by involved geometries due to the presence of fractures and faults,
see for instance~\cite{Benedetto2014,Fumagalli2016a,Benedetto2016,Benedetto2017,Fumagalli2017a}.

In this paper we are modelling a single-phase flow in porous media via a mixed formulation, i.e.,
we consider both the macroscopic (or Darcy) velocity and the pressure as unknowns.
Although we have to solve a linear system with a saddle-point structure,
such formulation has several advantages.
Indeed, it gives a velocity field that is locally mass conservative and thus perfectly suited for \as{the subsequent simulation of} other processes,
such as transport phenomena.
It is also suited for applications with strong variations in the coefficients, \sout{like the}, \as{typically} permeability,
see, e.g.,\cite{Raviart1977,Brezzi1985,Brezzi1987,Roberts1991,Boffi2013} and
\cite{BeiraodaVeiga2014b,Brezzi2014,BeiraoVeiga2016,Fumagalli2016a,Benedetto2017,Fumagalli2017a}
in the context of FEM and VEM, respectively.
Moreover, the virtual element spaces exploited to discretize the Darcy velocity are polynomial over the faces,
while in the primal formulation the unknown function is a virtual function~\cite{BeiraodaVeiga2017}.
In the following sections we will see that this is {the} key point we exploit
to incorporate in the functional spaces the exact geometry \as{of the domain and/or the cells boundaries}.

% Since we are dealing with virtual methods, another advantage of the mixed formulation
% is that the single-phase flow solved with virtual element method in primal form
% is rather complicated to implement in three space dimensions since, in addition
% to the volumetric virtual function, on each face an extra bi-dimensional virtual
% functions has to be computed, as in \cite{BeiraodaVeiga2017}. For the mixed formulation only the macroscopic
% velocity belongs to a virtual space, which has only volumetric virtual
% functions.

The main contribution of this work is to introduce, analyze and validate the Mixed Virtual Element Method (MVEM) for problems characterized by curved boundaries or internal curved interfaces.
 \as{The piece-wise planar approximation of these surfaces, which is often used in practice with traditional methods,}  introduces a geometrical error
that corrupts the optimal error decay, especially in the high order case.
To \sout{achieve this goal} \as{obtain a high order approximation with curved geometries}, we proceed on two fronts.
On one hand we extend the quadrature rule for curved polygons introduced  in~\cite{Sommariva2009,Sommariva2015,BeiraodaVeiga2019} to deal with polyhedrons characterized by curved faces.
On the other hand we extend the strategy proposed in~\cite{Dassi2020} for mixed VEM in 2D to the three dimensional case.

% Our strategy is to map the face degrees of freedom of the vector field from the
% physical space to the parametric space and there consider a polynomial
% approximation. The theory, developed mostly in two-dimensions in
% \cite{Dassi2020}, is here extended
% in some components, shows optimal properties of the scheme and the numerical
% examples studied show its potential in some applicative contexts. Finally,
% quadrature rules for curved polygons introduced in  are here extended to
% the three-dimensions.

This paper is organized as follows.
We introduce some notation useful for the rest of the work in Section~\ref{sec:notations}.
Then, in Section~\ref{sec:problem}, we present the mathematical problem in its strong and weak formulation.
Section~\ref{sec:vemSpaces} contains the detailed description of the VEM spaces and projection operators.
We extend the two stability analyses carried out in ~\cite{Dassi2020} for the 2D case to the three dimensional case in Section~\ref{sec:stability}.
In Section~\ref{sec:integration} we describe the quadrature rules exploited to integrate function in polyhedrons characterized by curved faces.
Numerical examples, inspired from real applications, are given in Section \ref{sec:numExe}.
Finally, in Section \ref{sec:conc}, we draw some conclusions.

%% file: notazioni.tex
\section{Notations and mesh assumptions}\label{sec:notations}

Throughout the paper, we follow the usual notation for Sobolev spaces and norms~\cite{Adams:1975}.
We employ the usual notations for gradient and Laplacian of scalar functions,
while $\diver$ denotes the divergence for vector fields.

Given a bounded domain $\omega \subset\mathbb{R}^3$, we indicate as
$\|{\cdot}\|_{W^s_p(\omega)}$ and $\|{\cdot}\|_{L^p(\omega)}$ the
norms in the spaces $W^s_p(\omega)$ and $L^p(\omega)$, respectively.
We denote the norm and semi-norm in $H^{s}(\omega)$ via $\|{\cdot}\|_{s,\omega}$
and $|{\cdot}|_{s,\omega}$. To keep the notation simpler, we adopt
$(\cdot,\cdot)_{\omega}$ and $\|\cdot\|_{\omega}$ to be the $[L^2(\omega)]^{d_1\times
d_2}$-inner product
and the $[L^2(\omega)]^{d_1\times d_2}$-norm, for $d_1\geq 1$ and $d_2 \geq 1$.
For the theoretical analysis of the present problem,
we use the following well-known functional spaces:
\begin{gather*}
    H(\diver, \omega) \defeq \{\bm{v} \in [L^3(\omega)]^3: \,   \diver
    \bm{v} \in L^2(\omega)\} \,,
    \\
    H(\curl, \omega) \defeq \{\bm{v} \in [L^3(\omega)]^3 :\,
    \curl \bm{v} \in [L^2(\omega)]^3\} \,,
\end{gather*}
with scalar products and induced norms indicated as $(\cdot,
\cdot)_{H(\diver, \omega)}$ and $\Vert{\cdot}\Vert_{H(\diver, \omega)}$, and $(\cdot,
\cdot)_{H(\curl, \omega)}$ and $\Vert{\cdot}\Vert_{H(\curl, \omega)}$.
For $\gamma \subset \partial \omega$ with normal $\bm{n}$, we can introduce the following
subspace of $H(\diver, \omega)$ as
\begin{gather*}
    H_\gamma(\diver, \omega) \defeq \{\bm{v} \in H(\diver, \omega):\,
    \bm{v}\cdot \bm{n} = 0\}.
\end{gather*}
Moreover we define $\P_k(\omega)$ the space of polynomials of degree lower or equal to $k$,
and $[\P_k(\omega)]^{d}$ the vector polynomial space of dimension $d>1$.
Moreover, given two quantities $a,b\in\mathbb{R}$,
we shall write ``$a \lesssim b$''
if there exists a positive constant $C$ independent
of the discretization parameters such that $a \leq C b$.

The results of this paper \as{rely} on standard polyhedral mesh assumptions~\cite{apollo}.
We suppose that for a domain $\Omega$ there exists a tessellation $\Omega_h$ that discretizes $\Omega$.
Such mesh is composed by general polyhedrons, $P$, that are star-shaped with respect to a ball of radius $r$.
We suppose that there exists a constant $\zeta$ such that
each face $F\in\partial P$ is star shaped with respect to a ball whose radius is bigger or equal to $\zeta h_P$,
where $h_P$ is the diameter of $P$.
We make a similar assumption for edges, i.e.,
we suppose that each edge of the skeleton of $P$ bigger or equal to $\zeta h_P$.
Starting from this assumptions we have the following relations
$h_P \lesssim h_F$ and $h_F \lesssim h_e$,
where $h_F$ is the diameter of a polyhedron face $F$, while
$h_e$ is the length of an edge $e$ of  polyhedron skeleton. \sout{Assumptions
related to curved faces will be present in the sequel}.\as{Based on these standard assumptions a suitable extension to the case of curved faces will be  presented in the next sections.}

%% file: problem.tex
\section{Model Problem}\label{sec:problem}

In this section, we introduce the mathematical model we are interested to solve.
We consider a three-dimensional domain $\Omega \subset \mathbb{R}^3$, whose
boundary $\partial \Omega$ is Lipschitz continuous with unit normal $\bm{n}$
pointing outward from $\Omega$. We stress the fact that the boundary, or a
portion of it, might be curved. To impose suitable boundary conditions, we
divide $\partial \Omega$ into two distinct parts: $\partial_e \Omega$ and
$\partial_n \Omega$, where in the former we set essential boundary conditions
and in the latter natural. We clearly have $\overline{\partial \Omega} =
\overline{\partial_e \Omega} \cup \overline{\partial_n \Omega}$ and
$\mathring{\partial_e \Omega} \cap \mathring{\partial_n \Omega} = \emptyset$.
Finally, for solvability issues we assume that $\mathring{\partial_n \Omega}
\neq \emptyset$.

The domain $\Omega$ represents a porous medium saturated by a single fluid (e.g.,
water). \as{The porous medium is characterized by} $\kappa$, the permeability tensor
assumed to be symmetric and positive defined, while the fluid is characterized
by a dynamic viscosity $\mu$ which is a strictly positive number. External
forces are the vector $\bm{g}$ and scalar $f$ sources,
$\bm{g}$ might represent a gravity term and $f$ an injection or extraction
well. On $\partial \Omega$ boundary conditions on the pressure $\overline{p}$ or
on the normal flux $\overline{q}$ might be imposed.

\as{Since we are considering a Darcy problem in mixed form} the unknowns are the Darcy velocity
$\bm{q}$ and the fluid pressure $p$, which are computed as in the following
problem.
\begin{problem}[Darcy problem - strong form]\label{pb:strong}
    Find $(\bm{q}, p)$ such that
    \begin{gather*}
        \left\{
        \begin{aligned}
            &\mu \bm{q} + \kappa \nabla p = \bm{g}\\
            &\diver \bm{q} = f
        \end{aligned}\right.
        \quad \text{in } \Omega
        \qquad
    \left\{
        \begin{aligned}
            &p = \overline{p} && \text{on } \partial_n \Omega\\
            &\bm{q} \cdot \bm{n} = \overline{q} && \text{on } \partial_e \Omega
        \end{aligned}
    \right.
    \end{gather*}
\end{problem}

To derive the weak form of Problem \ref{pb:strong}, we introduce the following
functional spaces for vector and scalar fields
\begin{gather*}
    \bm{V} \defeq H_{\partial_e \Omega}(\diver, \Omega)
    \quad \text{and} \quad
    Q \defeq L^2(\Omega).
\end{gather*}
We indicate the scalar products and induced norms of these spaces as: $(\cdot,
\cdot)_{\bm{V}}$ and $\Vert \cdot \Vert_{\bm{V}}$, and $(\cdot, \cdot)_Q$ and
$\Vert \cdot \Vert_Q$. By setting $\lambda = \mu \kappa^{-1}$, we introduce
the following forms
\begin{gather*}
    \begin{aligned}
        &a: \bm{V}\times\bm{V}\rightarrow\mathbb{R}
        &&a(\bm{u}, \bm{v}) \defeq (\lambda \bm{u}, \bm{v})_\Omega
        && \forall (\bm{u}, \bm{v}) \in \bm{V} \times \bm{V}\\
        &b: \bm{V} \times Q \rightarrow \mathbb{R}
        &&b(\bm{u}, v) \defeq - (\diver \bm{u}, v)_\Omega
        &&\forall (\bm{u}, v) \in \bm{V} \times Q
    \end{aligned}.
\end{gather*}
Moreover, we introduce the functionals as
\begin{gather*}
    \begin{aligned}
        &G: \bm{V} \rightarrow \mathbb{R}
        && G(\bm{v}) \defeq - (\overline{p}, \bm{v}\cdot\bm{n})_{\partial_n
        \Omega} + (\bm{g}, \bm{v})_\Omega
        && \forall \bm{v} \in \bm{V}\\
        &F: Q \rightarrow \mathbb{R}
        && F(v) \defeq -(f, v)_\Omega
        && \forall v \in Q
    \end{aligned}.
\end{gather*}
For the data, we assume in addition that $\kappa \in [L^\infty(\Omega)]^{3\times
3}$, $\mu \in L^2(\Omega)$, $\overline{p} \in H^{\frac{1}{2}}_{00}(\partial_n
\Omega)$, $\bm{g}\in[L^2(\Omega)]^3$, and $f \in L^2(\Omega)$. With this in
place, we can introduce the weak form of Problem  \ref{pb:strong} by assuming
$\overline{q}=0$.
\begin{problem}[Darcy problem - weak form]\label{pb:weak}
    Find $(\bm{q}, p) \in \bm{V}\times Q$ such that
    \begin{gather*}
        \left\{
        \begin{aligned}
            &a(\bm{q}, \bm{v}) + b(\bm{v}, p) = G(\bm{v})
            &&\forall \bm{v}\in \bm{V}\\
            &b(\bm{q}, v) = F(v)
            &&\forall v \in Q
        \end{aligned}\right..
    \end{gather*}
\end{problem}
Following, e.g., \cite{Boffi2013} it is possible to show that Problem
\ref{pb:weak} is well-posed.

\begin{remark}
    We have used an abuse in notation for the term $(\overline{p},
    \bm{v}\cdot\bm{n})_{\partial_n\Omega}$, in fact the fist entry is an element
    of $H^{\frac{1}{2}}_{00}(\partial_n\Omega)$ and the second of
    $H^{-\frac{1}{2}}(\partial_n\Omega)$, being the latter a trace of a
    $H(\diver, \Omega)$ function. To be precise, we should use a duality pairing
    between them and write $
        \langle \overline{p},
        \bm{v}\cdot\bm{n}\rangle_{H^{\frac{1}{2}}_{00}(\partial_n\Omega) \times H^{-\frac{1}{2}}(\partial_n\Omega)}
    $
    in place of the $L^2$-scalar product.
\end{remark}

%% file: vemSpaces.tex
\section{Virtual Element Spaces}\label{sec:vemSpaces}

To discretize Problem~\ref{pb:weak}, when the computational domain presents
curved boundaries or interfaces, we extend the discrete spaces proposed
in~\cite{BeiraodaVeiga2014b} via the idea in~\cite{BeiraodaVeiga2019}.
In Subsection \ref{subsec:approximation_spaces}, we give the motivations behind the proposed theory,
then we describe the local spaces, in Subsections
\ref{subsec:velocity_approximation} and \ref{subsec:pressure_approximation}
respectively, and the local operators
used to solve Problem~\ref{pb:weak} in Subsection \ref{subsec:problem_operators}.

\subsection{Approximation spaces}\label{subsec:approximation_spaces}

The idea behind the construction of the approximation spaces is simple and effective:
a virtual function restricted on a curved face is still a polynomial {but} in the parameter space.
The geometry is plugged in the virtual element space so
it is perfectly matched by any virtual function.
In~\cite{BeiraodaVeiga2019,Dassi2020,Dassi2021a} the authors give both the theoretical and numerical evidence about this nice property.

As it was done for the spaces in the aforementioned works,
our definition of the local functional spaces assumes that
there exists a sufficiently regular map $\gamma:\mathfrak{F}\subset[0,\,1]^2\to F$
that represents the geometry of the curved spaces.
$\widetilde{\P}_k({F})$ is given by
\begin{gather*}
    \widetilde{\P}_k({F}) := \{\widetilde{v}_h = v_h \circ \gamma^{-1}:
    v_h \in \P_k(\mathfrak{F})\}.
\end{gather*}
For simplicity, we assume that a curved face $F$ is a graph of a known regular
function. So, $\widetilde{\P}_k({F})$ is the space of polynomial of degree $k$ in the parameter space of the surface that defines the curved face.

The main difference with the classic approach is the characterization of the virtual function on the boundary of the element $E$.
Indeed, the virtual function is {always} a polynomial in the physical space for~\cite{BeiraodaVeiga2014b}.
Whereas, when a curved face is present in the current approach a function is still a polynomial in the physical space along straight edges
but it is a polynomial in the {parameter space} for curved boundary.

\begin{remark}
    A standard virtual element space used to solve a Laplacian problem has
    virtual functions on each polyhedron face~\cite{BeiraodaVeiga2014a,apollo}.
    Having Problem \ref{pb:weak} in mixed form avoids this difficulty.
\end{remark}

\subsection{Velocity approximation}\label{subsec:velocity_approximation}

Given a polyhedron $P$ that can have one or more curved faces,
we define the local space $\bm{\mathcal{V}}_h^k(P)$ as
\begin{align}\label{eqn:mixedCurved}
    \begin{aligned}
        \bm{\mathcal{V}}_h^k(P) :=  \big\{\vv_h\in H(\diver\,;\,P)\cap H(\curl\,;\,P)\::\:
        &\diver\vv_h\in\P_{k-1}(E),\\
        &\curl(\vv_h)\in[\P_{k-1}(E)]^3,\\
        &\vv_h\cdot\nn_F\in\P_k(F)\:\forall F\in\partial_{\mathcal{S}} P,\\
        &\vv_h\cdot\nn_F\in\widetilde{\P}_k({{F}})\:\forall F\in\partial_{\mathcal{C}} P\big\},
    \end{aligned}
\end{align}
where $\partial_{\mathcal{S}} P$ and $\partial_{\mathcal{C}} P$ denote the sets
of straight and curved faces, respectively.
Then, to have a discrete approximation of the velocity variable ${\bm q}$,
we use the global space
\begin{gather*}
    \bm{\mathcal{V}}_h^k(\Omega_h) := \{\vv_h\in H(\diver\,;\,\Omega_h)\cap H(\curl\,;\,\Omega_h)\::\:\vv_h|_P\in\bm{\mathcal{V}}_h^k(P)\:\forall P\in\Omega_h\}\,.
\end{gather*}
In each local space $\bm{\mathcal{V}}_h^k(P)$ we consider the following degrees of freedom
to uniquely identify a function $\vv_h \in \bm{\mathcal{V}}_h^k(P)$:
\begin{itemize}
 \item \textbf{normal face moments:}
 \begin{gather}\label{eqn:faceStMom}
    \frac{1}{|F|}\int_{F}(\vv_h \cdot\nn_F)\,m_k\dF\quad\forall m_k\in\P_k(F)\,,
 \end{gather}
 for each face $F\in\partial_{\mathcal{S}} P$, where $|F|$ is the area of the face $F$, and
 \begin{gather}\label{eqn:faceCvMom}
    \frac{1}{|\mathfrak{F}|}\int_{\mathfrak{F}}(\vv_h\cdot\nn_{\mathfrak{F}})\,{\mathfrak{m}_k}\ds\quad\forall {\mathfrak{m}}_k\in\widetilde{\P}_k({\mathfrak{F}})\,.
 \end{gather}
 for each face $F\in\partial_{\mathcal{C}} P$, where now $|\mathfrak{F}|$ is the area of the image of $F$ in the parameter space.
 \item \textbf{divergence moments:}
 \begin{gather*}
     \frac{h_P}{|P|}\int_P \diver(\vv_h)\, m_{k-1}\dP\,,\quad\forall m_{k-1}\in\P_{k-1}(P)\backslash\P_0(P)\,,
 \end{gather*}
 where $h_P$ and $|P|$ are the diameter and the volume of the polyhedron $P$.
 \item \textbf{internal cross moments:}
 \begin{gather*}
  \frac{1}{|P|}\int_P \vv_h\cdot(\mm_I\wedge\mm_{k-1})\dP\,,\quad\forall \mm_{k-1}\in[\P_{k-1}(P)]^3\,,
 \end{gather*}
 where $\mm_I:=(x,\,y,\,z)^\top$ and once again $|P|$ is the volume of the polyhedron.
\end{itemize}
A function in the global space $\bm{\mathcal{V}}_h^k(\Omega_h)$ is determined by
the union of such local d.o.f.

In the previous definitions we make the distinction between degrees of freedom associated with straight and curved faces
to highlight the presence of curved faces.
However, one can consider the degrees of freedom~\eqref{eqn:faceCvMom} for straight faces too.
Indeed, if we define a proper affine map
that represents the plane where the straight face lies on
so that we have the same types of degrees of freedom.
We can thus simplify \eqref{eqn:mixedCurved}
\begin{align*}
 \begin{aligned}
    \bm{\mathcal{V}}_h^k(P) := \big\{\vv\in H(\diver\,;\,P)\cap H(\curl\,;\,P)\::\:
    &\diver\vv_h\in\P_{k-1}(E),\\
    &\curl(\vv_h)\in[\P_{k-1}(E)]^3,\\
    &\vv_h\cdot \nn_F \in\widetilde{\P}_k({{F}})\:\forall F\in\partial P\big\}\,.
 \end{aligned}
\end{align*}
Such choice is useful from the computational point of view too.
Indeed, the only difference between straight and curved faces lies on the definition of the map $\gamma$.

\subsection{Pressure approximation}\label{subsec:pressure_approximation}

The global pressure variable $p$ is approximated via element-wise polynomials of degree $k-1$, i.e.,
\begin{gather*}
    Q_h(\Omega_h):=\left\{v\in L^2(\Omega_h):v|_P\in\P_{k-1}(P)\:\forall P\in\Omega_h\right\}\,.
\end{gather*}
To uniquely determine one polynomial on each polyhedron,
Given a polynomial $v \in \P_{k-1}(P)$,
its degrees of freedom are
\begin{itemize}
 \item \textbf{pressure moments:}
 \begin{gather*}
    \frac{1}{|P|}\int_P v \,m_{k-1}\dP\quad\forall m_{k-1}\in\P_{k-1}(P)\,.
 \end{gather*}
 \end{itemize}
A function of the global space $Q_h(\Omega_h)$ is identified by the union
of such local d.o.f.

\subsection{Problem operators}\label{subsec:problem_operators}

It is possible to
define a proper $L^2$ projection operator
${\bm\Pi}_k^0:\bm{\mathcal{V}}_h^k(P)\to [\P_k(P)]^3$, for $\vv \in
\bm{\mathcal{V}}_h^k(P)$
\begin{gather*}
    \int_P \left({\bm\Pi}_k^0\vv-\vv\right)\cdot\mm_k\dP = 0\quad\forall
    \mm_k\in[\P_k(P)]^3\,.
\end{gather*}
\af{Since $a(\cdot,\cdot)$ is not directly computable, we define its discrete
conterpart by the following decomposition
\begin{gather*}
    a_h(\vv,\, \ww) = a({\bm\Pi}_k^0 \vv,\, {\bm\Pi}_k^0 \ww) +
    \mathcal{S}^P({\bm T}_k^0\vv,\,{\bm T}_k^0\ww),
\end{gather*}
with the operator ${\bm T}_k^0 = {\bm I} - {\bm\Pi}_k^0$.}
The former \af{decomposition of $a(\cdot,\cdot)$} is
based on ${\bm\Pi}_k^0$ and a symmetric positive-definite bi-linear form
$\mathcal{S}^P$\af{, usually called stabilization}.
We choose
\begin{gather}\label{eqn::stab}
\mathcal{S}^P(\vv,\,\ww) := \|\nu\|_{L^{\infty}(P)}|P|\sum_{i=1}^{\#\texttt{dof}}\texttt{dof}_i\big(\vv\big)\texttt{dof}_i\big(\ww\big)\,,
\end{gather}
where $\#\texttt{dof}$ is the number of degrees of freedom associated with the polyhedron $P$ and
we define the function $\texttt{dof}_i$ that given a function in $\bm{\mathcal{V}}_h^k(P)$ returns its $i^{\text{th}}$ degrees of freedom.
The form $b(\cdot,\cdot)$ is directly computable given the chosen set of degrees of
freedom and it is exact since it involves only polynomials. The definition of such ingredients to have a virtual element version of Problem~\ref{pb:weak}
is deeply described in literature, see, e.g.,~\cite{Brezzi2014,Benedetto2017,BeiraodaVeiga2014b,Dassi2019a}.
The proposed scheme follows the same strategy and there are any further troubles.
As for the approach proposed in~\cite{BeiraodaVeiga2019},
the issue is related {only} on how make integration over curved domains.
\af{The discrete version of Problem \ref{pb:weak} is written as follow.
\begin{problem}[Darcy problem - discrete weak form]\label{pb:discrete_weak}
    Find $(\bm{q}, p) \in \bm{\mathcal{V}}_h^k(\Omega_h)\times Q_h(\Omega_h)$ such that
    \begin{gather*}
        \left\{
        \begin{aligned}
            &a_h(\bm{q}, \bm{v}) + b(\bm{v}, p) = G(\bm{v})
            &&\forall \bm{v}\in \bm{\mathcal{V}}_h^k(\Omega_h)\\
            &b(\bm{q}, v) = F(v)
            &&\forall v \in Q_h(\Omega_h)
        \end{aligned}\right..
    \end{gather*}
\end{problem}
}

%% file: stability.tex
\section{Stability of the scheme}\label{sec:stability}

In~\cite{Dassi2020} the authors already describe these theoretical aspects for the two dimensional case.
Similar arguments can be exploited for the three dimensional case too.
We limit ourself showing the three dimensional counterpart the stability of the
bilinear discrete operator $a_h^P(\cdot,\cdot)$\af{, restriction of
$a_h(\cdot,\cdot)$ to the element $P$}.
To achieve this goal,
we need the following additional results.

%%%%%%%%%%%%%%%%%%%%%%%%%%%%%%%%%%%%%%%%%%%%%%%%%%%%%%%%%%%%%%%%%%%%%%%%%%%%%%%%%%%%%%%%%%%%%%%%%%%%%%%%%%%%%%%%%%%%%%%%%%%%%%%%%%%%%%%%%%
%%%%%%%%%%%%%%%%%%%%%%%%%%%%% Lemma per la divergenza %%%%%%%%%%%%%%%%%%%%%%%%%%%%%%%%%%%%%%%%%%%%%%%%%%%%%%%%%%%%%%%%%%%%%%%%%%%%%%%%%%%%
%%%%%%%%%%%%%%%%%%%%%%%%%%%%%%%%%%%%%%%%%%%%%%%%%%%%%%%%%%%%%%%%%%%%%%%%%%%%%%%%%%%%%%%%%%%%%%%%%%%%%%%%%%%%%%%%%%%%%%%%%%%%%%%%%%%%%%%%%%
\begin{lemma}\label{lem:divEst}
    Let $\ww\in H(\diver,\,P)$ such that $\diver\ww\in\mathbb{P}_{k-1}(E)$ then
    \begin{gather*}
        \|\diver\ww\|_{0,P} \lesssim h_P^{-1} \|\ww\|_{0,P}\,,
    \end{gather*}
    where $h_P$ is the diameter of the polyhedron $P$, moreover for $\ww \in
    H(\curl, P)$ such that $\curl \ww \in [\mathbb{P}_{k-1}(E)]^3$ we have
    \begin{gather*}
        \|\curl\ww\|_{0,P} \lesssim h_P^{-1} \|\ww\|_{0,P}.
    \end{gather*}
\end{lemma}
\begin{proof}
    Since the polyhedron $P$ is star-shaped there exists a sphere $B_P$ that contains $P$ and
    there exists also a regular tetrahedron, $T_P$, inscribed in the sphere $B_P$, such that
    for any $p_k\in\P_k(P)$ we have
    $\|p_k\|_{0,P}\lesssim \|p_k\|_{0,T_P}$.
    Let $b_4\in\P_4(T_P)$ be the bubble such that $\|b_4\|_{L^\infty(T_P)}<1$ and
    that is identically zero on $\partial P$.
    Then, since the divergence of any function in $\ww\in\bm{\mathcal{V}}_h^k(P)$  is a polynomial of degree $k-1$,
    we have
    \begin{align*}
        \|\diver\ww\|_{0,P}^2 &\lesssim \|\diver\ww\|_{0,T_P}^2                                 && \\
        &\lesssim \int_P b_4\,\diver\ww\,\diver\ww\dP                                           &&(b_4\text{ is limited})\\[0.5em]
        &\lesssim -\int_P \nabla(b_4\,\diver\ww)\cdot\ww\dP                                       &&(\text{integration by parts})\\[0.5em]
        &\lesssim \|\nabla(b_4\,\diver\ww)\|_{0,T_P}\,\|\ww\|_{0,T_P}                             &&(\text{H\"older})\\[0.5em]
        &\lesssim h_P^{-1}\|b_4\,\diver\ww\|_{0,T_P}\,\|\ww\|_{0,T_P}                             &&(H^1\text{ inverse inequality})\\[0.5em]
        &\lesssim h_P^{-1}\|b_4\|_{L^\infty(T_P)}\,\|\diver\ww\|_{0,T_P}\,\|\ww\|_{0,T_P}         &&(b_4\text{ is limited})\\[0.5em]
        &\lesssim h_P^{-1}\,\|\diver\ww\|_{0,T_P}\,\|\ww\|_{0,T_P}                                &&(\text{mesh assumption})\\[0.5em]
        &\lesssim h_P^{-1}\,\|\diver\ww\|_{0,P}\,\|\ww\|_{0,P}\,.
    \end{align*}
    Looking at the beginning and at the end of the previous chain of inequalities,
    we have the proof. For bound on the $\curl$ follows the same strategy.
\end{proof}
%%%%%%%%%%%%%%%%%%%%%%%%%%%%%%%%%%%%%%%%%%%%%%%%%%%%%%%%%%%%%%%%%%%%%%%%%%%%%%%%%%%%%%%%%%%%%%%%%%%%%%%%%%%%%%%%%%%%%%%%%%%%%%%%%%%%%%%%%%
%%%%%%%%%%%%%%%%%%%%%%%%%%%%% Lemma per i polinomi e i suoi coefficienti %%%%%%%%%%%%%%%%%%%%%%%%%%%%%%%%%%%%%%%%%%%%%%%%%%%%%%%%%%%%%%%%%
%%%%%%%%%%%%%%%%%%%%%%%%%%%%%%%%%%%%%%%%%%%%%%%%%%%%%%%%%%%%%%%%%%%%%%%%%%%%%%%%%%%%%%%%%%%%%%%%%%%%%%%%%%%%%%%%%%%%%%%%%%%%%%%%%%%%%%%%%%

\begin{lemma}\label{lem:coef}
    A polynomial $g\in \P_k(P)$, of degree $k$,
    can be written as
    \begin{gather}\label{eqn:monoDeco}
        g(\xx) = \sum_{r=1}^{\pi_k} g_r\,m_r(\xx)\,,
    \end{gather}
    where $\{m_r\}_{r=1}^{\pi_k}$ are so-called scaled monomials of degree lower
    or equal to $k$.
    Such set is a basis of polynomials of degree $k$ and then the following inequalities hold
    \begin{gather}\label{eqn:coeff}
        h_P^3\|\gg\|_{l^2}^2\lesssim \|g\|^2_{0,P}\lesssim h_P^3\|\gg\|_{l^2}^2\,,
    \end{gather}
    where $\gg\in\R^{\pi_k}$ is the vector of monomials' coefficients in~\eqref{eqn:monoDeco},
    $\pi_k$ is the dimension of the space $\P_k(P)$ and
    $\|\cdot\|_{l^2}$ is the standard Euclidean norm.
\end{lemma}
\begin{proof}
    The proof of this lemma is an extension to the three dimensional case of the one presented in~\cite{chen2018some}.
    Let us prove the second part of \eqref{eqn:coeff}.
    \begin{align*}
        \|g\|_{0,P}^2 &:= \int_P g^2\dP = \int_P\left(\sum_{r=1}^{\pi_k}g_r\,m_r\right)^2\dP                           &&(\text{definitions}) \\
                      &\leq \sum_{r=1}^{\pi_k} \int_P g_r^2\,m_r^2\dP = \sum_{r=1}^{\pi_k} g_r^2 \int_P \,m_r^2\dP     &&(\text{integral properties})\\
                      &\lesssim \sum_{r=1}^{\pi_k}
                      g_r^2\,h_P^3\,\|m_r^2\|_{L^\infty(P)}  \lesssim
                      h_P^3\|\gg\|_{l^2}^2
                      &&(m_r\text{ is limited.})
    \end{align*}
    Let us move to the first part. Thanks to mesh assumptions, we can decompose a polyhedron $P$ in tetrahedrons, $\mathcal{T}_P$. Given one of these tetrahedrons, $\tau$, we can choose a sphere $S_\tau$ such that its radius satifies $r_\tau=\delta_\tau h_P$, where $\delta_\tau\in (0,1)$.
    Now we can apply the affine map
    \begin{equation}
    \widetilde{\xx} = \frac{\xx-\xx_P}{h_P}\,,
    \label{eqn:map}
    \end{equation}
    where $\xx_P$ is the barycenter of $P$. By construction, such map transforms $S_\tau$ in a sphere of radius $\delta_\tau$ and whose center is inside the unit sphere centrend in the origin, $\widetilde{S}_\tau$.
    Since $S_\tau\subset P$ we have
    \begin{equation}
    \|g\|_{0,P}\geq \|g\|_{0,S_\tau} =  \|\widetilde{g}\|_{0,\widetilde{S}_\tau}\,(h_P)^{3/2}\,,
    \label{eqn:firstPartCoeff}
    \end{equation}
    where $\widetilde{g}$ is the transformation of $g$ by the map of \eqref{eqn:map}.
    The last term becomes
    $$
    \|\widetilde{g}\|_{0,\widetilde{S}_\tau}^2 = \gg^t \widetilde{M}^\tau \gg\,,
    $$
    where $\widetilde{M}$ is a matrix whose entries are
    $$
    \widetilde{M}_{i,j}^\tau = \int_{\widetilde{S}_\tau} \widetilde{m}_i\,\widetilde{m}_j~\text{d}\widetilde{\xx}\,.
    $$
    The matrix $\widetilde{M}^\tau$ is a depends on the position of the center
    $\bm{c}$ of the sphere $S_\tau$, then
    $$
    \sum_{\tau\in\mathcal{T}_P}\|\widetilde{g}\|^2_{0,\widetilde{S}_\tau}\geq \lambda^* \|\gg\|_{l^2}^2\,,
    $$
    where $\lambda^*$ is the minimum eigenvalue among all transformations $M^\tau$. We underline that $\lambda^*$ does not depend on the mesh size $h_P$ but {only} on the radii $\delta_\tau$ so we can write
    \begin{equation}
    \sum_{\tau\in\mathcal{T}_P}\|\widetilde{g}\|^2_{0,\widetilde{S}_\tau}\gtrsim \|\gg\|_{l^2}^2\,
    \label{eqn:secondPartCoeff}
    \end{equation}
    Now, squaring \eqref{eqn:firstPartCoeff} and exploitng the estimate of
    \eqref{eqn:secondPartCoeff}, we have the desired estimate.%, $\|g\|^2_{0,P}\gtrsim h_P^3\|\gg\|_{l^2}^2$.
\end{proof}
%%%%%%%%%%%%%%%%%%%%%%%%%%%%%%%%%%%%%%%%%%%%%%%%%%%%%%%%%%%%%%%%%%%%%%%%%%%%%%%%%%%%%%%%%%%%%%%%%%%%%%%%%%%%%%%%%%%%%%%%%%%%%%%%%%%%%%%%%%
%%%%%%%%%%%%%%%%%%%%%%%%%%%%%% Lemma per le stime che serviranno per il lower bound %%%%%%%%%%%%%%%%%%%%%%%%%%%%%%%%%%%%%%%%%%%%%%%%%%%%%%
%%%%%%%%%%%%%%%%%%%%%%%%%%%%%%%%%%%%%%%%%%%%%%%%%%%%%%%%%%%%%%%%%%%%%%%%%%%%%%%%%%%%%%%%%%%%%%%%%%%%%%%%%%%%%%%%%%%%%%%%%%%%%%%%%%%%%%%%%%
\begin{lemma}\label{lem:lemEstOk}
    Given a polyhedron $P$ and a function $\ww\in\bm{\mathcal{V}}_h^k(P)$ it
    holds
    \begin{align}
        \big(\ww\cdot\nn\,,\,\mathfrak{m}_k\big)_{\mathfrak{F}} &\lesssim h_P^{1/2} \big(\mathcal{S}^P(\ww,\,\ww)\big)^{1/2} &&\text{if }\deg(\mathfrak{m}_k)\leq k\,,\label{eqn:primaL1}\\
        \big(\diver\ww\,,\,m_{k-1}\big)_{P} &\lesssim h_P^{1/2} \big(\mathcal{S}^P(\ww,\,\ww)\big)^{1/2} &&\text{if }\deg(m_{k-1})\leq k-1\,,\label{eqn:secondaL1}\\
        \big(\ww\,,\,\mm_I\wedge\mm_{k-1}\big)_{P} &\lesssim h_P^{1/2} \big(\mathcal{S}^P(\ww,\,\ww)\big)^{1/2} &&\text{if }\deg(\mm_{k-1})\leq k-1\label{eqn:terzaL1}\,.
    \end{align}
\end{lemma}
\begin{proof}
    We make the proof for each of the previous inequalities,
    starting from~\eqref{eqn:primaL1}.
    If $\deg(\mathfrak{m}_k) \leq k$ there exists a degrees of freedom associate with such monomial and the face $\mathfrak{F}$,
    we call it $\texttt{D1}_{\underline{i}}(\ww)$
    that
    \begin{align*}
       \big(\ww\cdot\nn\,,\,\mathfrak{m}_k\big)_{\mathfrak{F}}^2  &= |\mathfrak{F}|^2\,\texttt{D1}^2_{\underline{i}}(\ww)       &&(\text{definition})\\
        &\lesssim h_P^4\,\texttt{D1}^2_{\underline{i}}(\ww)\lesssim h_P\,|P|\,\texttt{D1}^2_{\underline{i}}(\ww)          &&(\text{mesh assumptions})\\
        &\lesssim h_P\,\mathcal{S}^P(\ww,\,\ww)                                                                           &&(\text{add the other terms}).
    \end{align*}
    Then, making the square root of both sides we get \eqref{eqn:primaL1}.

    \medskip For~\eqref{eqn:secondaL1} we distinguish two cases:
    first we consider the case where $m_{k-1}=1$, then  $1<\deg(m_{k-1})\leq k-1$.
    In the former case we proceed as
    \begin{align*}
        \big(\diver\ww\,,1\big)_{P}^2 &= \sum_{F\in\partial P} \|(\ww\cdot\nn)\|^2_{1,F}  &&(\text{divergence theorem}) \\
        &\lesssim h_P\,\mathcal{S}^P(\ww,\,\ww)
        &&(\text{use \eqref{eqn:primaL1}}).
    \end{align*}
    As before, by making the square root of both sides we get \eqref{eqn:secondaL1} when we consider the constant polynomial.
    Then, if $m_{k-1}$ is not the constant monomial,
    we know that there exists a degrees of freedom associated with such monomial,
    we call it $\texttt{D2}_{\underline{i}}(\ww)$
    \begin{align*}
        \big(\diver\ww\,,\,m_{k-1}\big)_{P}^2 &= \frac{|P|^2}{h_P^2}\,\texttt{D2}^2_{\underline{i}}(\ww) &&(\text{definition})\\
        &\lesssim h_P^4\,\texttt{D2}^2_{\underline{i}}(\ww)\lesssim h_P|P|\,\texttt{D2}^2_{\underline{i}}(\ww) &&(\text{mesh assumptions})\\
        &\lesssim h_P\,\mathcal{S}^P(\ww,\,\ww)                                                               &&(\text{add the other terms}).
    \end{align*}
    Finally, making the square root of both sides we got the relation in \eqref{eqn:secondaL1} for any momomials whose degree is greater than 0 and lower than~$k$.

    \medskip Also for~\eqref{eqn:terzaL1} we know that there exists a degrees of freedom associated with the monomial $\mm_{k-1}$ or
    at least a linear combination of such vectorial monomials that equals to $\mm_{k-1}$.
    We prove the former case, the latter is straightforward.
    \begin{align*}
        \big(\ww\,,\,\mm_I\wedge\mm_{k-1}\big)_{P}^2 &= |P|^2\,\texttt{D3}^2_{\underline{i}}(\ww) &&(\text{definition})\\
        &\lesssim h_P\,\mathcal{S}^P(\ww,\,\ww)                                            &&(\text{add the other terms}).
    \end{align*}
    We obtain the estimate of \eqref{eqn:terzaL1} by
    making the square root of both sides.
\end{proof}

%%%%%%%%%%%%%%%%%%%%%%%%%%%%%%%%%%%%%%%%%%%%%%%%%%%%%%%%%%%%%%%%%%%%%%%%%%%%%%%%%%%%%%%%%%%%%%%%%%%%%%%%%%%%%%%%%%%%%%%%%%%%%%%%%%%%%%%%%%
%%%%%%%%%%%%%%%%%%%%%%%%%%%%%% Proposizione per l'upper bound %%%%%%%%%%%%%%%%%%%%%%%%%%%%%%%%%%%%%%%%%%%%%%%%%%%%%%%%%%%%%%%%%%%%%%%%%%%%
%%%%%%%%%%%%%%%%%%%%%%%%%%%%%%%%%%%%%%%%%%%%%%%%%%%%%%%%%%%%%%%%%%%%%%%%%%%%%%%%%%%%%%%%%%%%%%%%%%%%%%%%%%%%%%%%%%%%%%%%%%%%%%%%%%%%%%%%%%
\begin{prop}
    The %symmetric positive-definite bi-linear
    form $\mathcal{S}^P$ defined in \eqref{eqn::stab} satisfies the following property
    \begin{gather}\label{eqn:estUp}
    \mathcal{S}^P(\ww,\,\ww) \lesssim \|\ww\|^2_{0,P}\quad\forall\ww\in\bm{\mathcal{V}}_h^k(P)\,.
    \end{gather}
\label{prop:estUp}
\end{prop}
\begin{proof}
$\mathcal{S}^P$ can be split into three parts associated with the three types of dofs,
\begin{align}\label{eqn:eachPart}
    \begin{aligned}
    \mathcal{S}^P(\ww,\,\ww) &:=   \|\nu\|_{L^{\infty}(P)}|P|\sum_{i=1}^{\#\texttt{D1}} \texttt{D1}_i^2\big(\ww\big)  &&(\text{normal face moments})\\
                             &+    \|\nu\|_{L^{\infty}(P)}|P|\sum_{i=1}^{\#\texttt{D2}} \texttt{D2}_i^2\big(\ww\big)  &&(\text{divergence moments})\\
                             &+    \|\nu\|_{L^{\infty}(P)}|P|\sum_{i=1}^{\#\texttt{D3}} \texttt{D3}_i^2\big(\ww\big)  &&(\text{internal cross moments}),
    \end{aligned}
\end{align}
where $\#\texttt{D1}$, $\#\texttt{D2}$ and $\#\texttt{D3}$ are the number of each type of dofs associated with each the polyhedron $P$.
Without loss of generality we consider the case $\nu\equiv1$.
To prove the relation in \eqref{eqn:estUp},
we show that the estimate hold for each terms of \eqref{eqn:eachPart}.
Then, \eqref{eqn:estUp} is a direct consequence of these relations.

\paragraph{Normal face moments}
We follow the remark made at the end of the ``velocity local space'' paragraph and we always consider a map $\gamma$ although the face is straight.
We fix a generic face dofs and
we prove the bound in \eqref{eqn:boundLocal} for this particular case.
\begin{align*}
|P|\,\texttt{D1}_i^2\big(\ww\big) &=
|P|\left(\frac{1}{|\mathfrak{F}|}\int_{\mathfrak{F}}(\ww\cdot\nn_F)\,{\mathfrak{m}_k}\ds\right)^2     &&(\text{definition})\\[0.5em]
&\lesssim \frac{|P|}{|\mathfrak{F}|^2}\,\|\ww\cdot\nn_F\|^2_{0,\mathfrak{F}}\, \|{\mathfrak{m}_k}\|^2_{0,\mathfrak{F}}                     &&(\text{Cauchy-Schwarz})\\[0.5em]
&\lesssim \frac{|P|}{|\mathfrak{F}|^2}\,\|\ww\cdot\nn_F\|^2_{0,\mathfrak{F}}\, |\mathfrak{F}|
\lesssim \frac{|P|}{|\mathfrak{F}|}\,\|\ww\cdot\nn_F\|^2_{0,\mathfrak{F}}                                                                  &&(\text{scaled monomials})\\[0.5em]
&\lesssim \frac{h_P^3}{h_P^2} \|\ww\cdot\nn_F\|^2_{0,F}\lesssim h_P\|\ww\cdot\nn_F\|^2_{0,F}                                                                  &&(\text{mesh assumpions})\\[0.5em]
&\lesssim h_P\|\ww\cdot\nn\|^2_{0,\partial P}
\end{align*}
Now following the same argumets proposed in the proof of Proposition~4.3
in~\cite{Dassi2020} and Lemma~\ref{lem:divEst},
we have
$$
h_P\|\ww\cdot\nn\|^2_{0,\partial P} \lesssim \|\ww\|_{0,P}^2 + h_P^2 \|\diver\ww\|_{0,P}^2 \lesssim \|\ww\|_{0,P}^2\,,
$$
and this complete the proof of the bound for a generic normal face moments.

\paragraph{Divergence moments}
Before dealing with the prove of this part,
we recall that the monomials involved on such degrees of freedom have degree greater than zero, i.e.,
there is not the constant monomial.
We fix a generic dof of this type and
then we prove the desired estimate.
\begin{align*}
|P|\,\texttt{D2}_i^2\big(\ww\big) &:=
|P|\left(\frac{h_P}{|P|}\int_P\diver(\ww)\,m_k \dP \right)^2                                            &&(\text{definition})\\[0.5em]
&= \frac{h^2_P}{|P|}\left(\int_P\diver(\ww)\,m_k\dP \right)^2                                                                                 &&\\[0.5em]
&\lesssim \frac{1}{h_P}\left(\int_P\diver(\ww)\,m_k\dP \right)^2                                                                              &&(\text{mesh assumpion})\\[0.5em]
&\lesssim \frac{1}{h_P}\|\diver\ww\|^2_{0,P}\,\|m_k\|^2_{0,P}                                                                           &&(\text{H\"older})\\[0.5em]
&\lesssim \frac{1}{h_P}\|\diver\ww\|^2_{0,P}\,|P|                                                                                       &&(m_k\text{ is limited})\\[0.5em]
&\lesssim \frac{1}{h_P^3}\|\ww\|^2_{0,P}\,|P|                                                                                             &&(\text{Lemma~\ref{lem:divEst}})\\[0.5em]
&\lesssim \frac{1}{h_P^3}\|\ww\|^2_{0,P}\,h_P^3 \lesssim \|\ww\|^2_{0,P}                                                                  &&(\text{mesh assumption})
\end{align*}

\paragraph{Internal cross moments}
We fix a generic dof of this type and get
\begin{align*}
|P|\,\texttt{D3}_i^2\big(\ww\big) &:= |P|\left(\frac{1}{|P|}\int_P
\ww\cdot(\mm_I\wedge\mm_{k-1})\dP\right)^2   &&(\text{definition})\\[0.5em]
&=\frac{1}{|P|}\left(\int_P \ww\cdot(\mm_I\wedge\mm_{k-1})\dP\right)^2                                           &&\\[0.5em]
&\lesssim \frac{1}{|P|} \|\ww\|^2_{0,P}\,\|\mm_I\wedge\mm_{k-1}\|^2_{0,P}                                     &&(\text{H\"older})\\[0.5em]
&\lesssim \frac{1}{|P|} \|\ww\|^2_{0,P}\,|P|                                                                  &&(\mm_I\wedge\mm_{k-1}\text{ is limited})\\[0.5em]
&\lesssim \frac{1}{h_P^3} \|\ww\|^2_{0,P}\,h_P^3 =  \|\ww\|^2_{0,P}                                           &&(\text{mesh assumption})
\end{align*}
%\vspace{2em}

\noindent We prove the estiamte for each kind of degrees of freedom.
Then
if we plug such estimate in \eqref{eqn:eachPart},
we get the estimate of \eqref{eqn:estUp}.
\end{proof}

%%%%%%%%%%%%%%%%%%%%%%%%%%%%%%%%%%%%%%%%%%%%%%%%%%%%%%%%%%%%%%%%%%%%%%%%%%%%%%%%%%%%%%%%%%%%%%%%%%%%%%%%%%%%%%%%%%%%%%%%%%%%%%%%%%%%%%%%%%
%%%%%%%%%%%%%%%%%%%%%%%%%%%%%% Proposizione per il lower bound %%%%%%%%%%%%%%%%%%%%%%%%%%%%%%%%%%%%%%%%%%%%%%%%%%%%%%%%%%%%%%%%%%%%%%%%%%%
%%%%%%%%%%%%%%%%%%%%%%%%%%%%%%%%%%%%%%%%%%%%%%%%%%%%%%%%%%%%%%%%%%%%%%%%%%%%%%%%%%%%%%%%%%%%%%%%%%%%%%%%%%%%%%%%%%%%%%%%%%%%%%%%%%%%%%%%%%
\begin{prop}\label{prop:estDown}
The form $\mathcal{S}^P$ defined in \eqref{eqn::stab} satisfies the following property
\begin{gather}\label{eqn:estDown}
\|\ww\|_{0,P} \lesssim \mathcal{S}^P(\ww,\,\ww),\quad\forall\ww\in\bm{\mathcal{V}}_h^k(P)\,.
\end{gather}
\end{prop}
\begin{proof}
We can decompose a function $\ww\in\bm{\mathcal{V}}_h^k(P)$ via the so-called
Helmholtz decomposition, as reported in Corollary 3.4 of \cite{Girault1986},
given by
$
\ww := \nabla \phi + \curl\AA
$,
such that $\phi$ and  $\AA$ solve the following problems
\begin{gather*}
    \begin{cases}
        \Delta \phi = -\diver \ww & \text{in } P \\
        \nabla \phi \cdot \nn_F = \ww \cdot \nn_F & \forall F \in \partial P
    \end{cases}
    \qquad
    \begin{cases}
        -\Delta \AA = \curl \ww  &\text{in } P\\
        \diver \AA = 0 & \text{in } P\\
        \AA \times \nn_F = \bm{0} & \forall F \in \partial P
    \end{cases}
\end{gather*}
To prove \eqref{eqn:estDown}, we consider each of these fields and we show
$$
\|\nabla \phi\|_{0,P}^2\lesssim \mathcal{S}^P(\ww,\,\ww)\quad\text{and}\quad
\|\curl\AA\|_{0,P}^2\lesssim \mathcal{S}^P(\ww,\,\ww)\,.
$$
Befere dealing with the first term, we recall both $\diver\ww$ and $(\ww\cdot\nn)$ are polynomials and
we can define a proper $L^2$ projection operators on polyhedrons and faces into the polynomial space, i.e., $\Pi_P\phi$ and $\widetilde{\Pi}_{\mathfrak{F}}\phi$.
Here $\Pi_P\phi$ refers to the Cartesian global space,
while $\widetilde{\Pi}_{\mathfrak{F}}\phi$ is an $L^2$ projection on the parameter space of the face $F$.
\begin{align*}
\|\nabla \phi\|_{0,P}^2 &=\int_P \nabla \phi\cdot\nabla \phi\dP = -\int_P
\phi\diver\ww\dP + \sum_{F\in\partial P} \int_F \phi\,(\ww\cdot\nn_F)\dF &&\\[0.5em]
&=-\int_P \Pi_P\phi\diver\ww\dP + \sum_{F\in\partial P} \int_{\mathfrak{F}}
\widetilde{\Pi}_{\mathfrak{F}}\phi\,(\ww\cdot\nn_\mathfrak{F})\ds&&\hspace{-7em}(\text{$L^2$ projection property})\\[0.5em]
&=-\sum_{i=1}^t c_i\int_P m_i\diver\ww\dP + \sum_{F\in\partial P} \sum_{j=1}^s d_j\int_{\mathfrak{F}} \mathfrak{m}_j\,(\ww\cdot\nn_{\mathfrak{F}} )\ds                                      &&\\[0.5em]
&&&\hspace{-12em}(\text{highlight polynomials' coefficients})\\[0.5em]
&\lesssim \sum_{i=1}^t c_i\,h_P^{1/2} \big(\mathcal{S}^P(\ww,\,\ww)\big)^{1/2} +
\sum_{F\in\partial P} \sum_{j=1}^s d_j\, h_P^{1/2} \big(\mathcal{S}^P(\ww,\,\ww)\big)^{1/2}&&\\[0.5em]
&&&\hspace{-2em}(\text{Lemma~\ref{lem:lemEstOk}})\\[0.5em]
&\leq\left( \|\bm{c}\|_{l^2}\,h_P^{1/2}
+ \sum_{F\in\partial P} \|\bm{d}\|_{l^2}\,h_P^{1/2} \right)\big(\mathcal{S}^P(\ww,\,\ww)\big)^{1/2}&&\\[0.5em]
&&&\hspace{-6em}\hspace{-5em}(\text{norm equivalence in $\mathbb{R}^t$ and $\mathbb{R}^s$})\\[0.5em]
&\lesssim \left(h_P^{-3/2}\|\Pi_P\phi\|_{0,P}\,h_P^{1/2}
+ \sum_{F\in\partial P}
h_P^{-1/2}\|\widetilde{\Pi}_{\mathfrak{F}}\phi\|_{0,{\mathfrak{F}}}\,h_P^{1/2}\right) \big(\mathcal{S}^P(\ww,\,\ww)\big)^{1/2} &&\\[0.5em]
&&&\hspace{-6em}(\text{Lemma~\ref{lem:coef} and~\cite{chen2018some}})\\[0.5em]
&\lesssim \left( h_P^{-1}\|\phi\|_{0,P} + \sum_{F\in\partial P}
\|\phi\|_{0,{\mathfrak{F}}}\right) \big(\mathcal{S}^P(\ww,\,\ww)\big)^{1/2} &&
\hspace{-9em}(\text{continuity of $\Pi_P$ and $\widetilde{\Pi}_{\mathfrak{F}}$})\\[0.5em]
&\lesssim \left(h_P^{-1}\|\phi\|_{0,P} + \sum_{F\in\partial P} \|\phi\|_{0,{\mathfrak{F}}} \right)\,(\mathcal{S}^P(\ww,\,\ww)\big)^{1/2}\\[0.5em]
&\lesssim \|\nabla \phi\|_{0,P}\,(\mathcal{S}^P(\ww,\,\ww)\big)^{1/2}\,. &&
\end{align*}
Coming back to the beginning of such chain of inequalities we have
$$
\|\nabla \phi\|_{0,P} \lesssim (\mathcal{S}^P(\ww,\,\ww)\big)^{1/2}\,.
$$
Let us consider the second part, since $\curl\ww = \curl(\nabla \phi + \curl\AA)
= \curl\curl\AA$  and that $\AA \times \nn_F = \bm{0}$ for all faces $F$, we get
\begin{align*}
\|\curl\AA\|_{0,P}^2 &= \int_P \curl\AA\cdot\curl\AA\dP  = \int_P \curl\ww \cdot \AA \dP
%&&\\[0.5em]
\end{align*}
Having $\ww \in \bm{\mathcal{V}}_h^k(P)$ this implies that
$\curl\ww\in[\mathbb{P}_{k-1}(P)]^3$ and also that it exists a $\bm{p} \in
[\mathbb{P}_{k-1}(P)]^3$ such that $\curl\ww = \curl(\bm{x} \wedge \bm{p})$. Moreover, by \cite{Droniou2021} we have that
\begin{gather*}
    \norm{\bm{x} \wedge \bm{p}}_{0, P} \lesssim h_P \norm{\curl{\ww}}_{0, P}
    \lesssim \norm{\curl{\AA}}_{0, P}
    \quad \text{and} \quad
    \bm{x} \wedge \bm{p} = \sum_{i=1}^{\pi_{k-1}} g_i \mm_I
    \wedge\mm_{i}
\end{gather*}
with $g_i \in \mathbb{R}$ are suitable coefficients,
and the latter inequality follows from Lemma \ref{lem:divEst}.
Continuing from before we get
\begin{align*}
    &\int_P \curl\ww \cdot \AA \dP = \int_P
    \curl(\bm{x} \wedge\bm{p}) \cdot \AA \dP =
    \int_P
    \bm{x} \wedge\bm{p} \cdot \curl\AA \dP =\\&
    \int_P
    \bm{x} \wedge \bm{p} \cdot (\ww - \nabla \phi)\dP =
    \int_P
    \bm{x} \wedge \bm{p} \cdot \ww\dP -
    \int_P
    \bm{x} \wedge \bm{p} \cdot \nabla \phi\dP.
\end{align*}
Let us consider the first term and realize that it is related to the third set
of degrees of freedom for the vector fields in $\bm{\mathcal{V}}_h^k(P)$, thus
we have
\begin{align*}
    &\int_P
    \bm{x} \wedge \bm{p} \cdot \ww\dP =
    \sum_{i=1}^{\pi_{k-1}} g_i \int_P \mm_I \wedge \mm_{i} \cdot \ww\dP = \abs{P}
    \sum_{i=1}^{\pi_{k-1}} g_i \texttt{D3}_i(\ww) \leq \\
    &h_P^{\frac{3}{2}}
    \left(\sum_{i=1}^{\pi_{k-1}} g_i^2 \right)^{\frac{1}{2}} \left(\abs{P}
    \sum_{i=1}^{\pi_{k-1}} \texttt{D3}_i(\ww)^2 \right)^{\frac{1}{2}} \leq
    \norm{\bm{x} \wedge \bm{p}}_{0, P} (\mathcal{S}^P(\ww,\,\ww)\big)^{1/2}
    \leq\\&
    \| \curl {\AA} \|_{0,P}
    (\mathcal{S}^P(\ww,\,\ww)\big)^{1/2}.
\end{align*}
We have now the term with $\nabla \phi$ to bound, namely
\begin{gather*}
    - \int_P
    \bm{x} \wedge \bm{p} \cdot \nabla \phi\dP \leq
    \| \bm{x} \wedge \bm{p} \|_{0,P}
    \| \nabla \phi \|_{0,P}\lesssim
    \| \curl {\AA} \|_{0,P}
    (\mathcal{S}^P(\ww,\,\ww)\big)^{1/2}.
\end{gather*}
Collecting the results we finally obtain the proof.
\end{proof}

\begin{prop}
Consider the bi-linear form
\begin{gather}
a^P_h(\vv,\ww) := \int_P \nu(x)\,{\bm\Pi}_k^0 \vv\cdot {\bm\Pi}_k^0 \ww\dP + \mathcal{S}^P(\vv,\,\ww)\,,
\label{eqn:biForm}
\end{gather}
where we defined
$$
 \mathcal{S}^P(\vv,\,\ww) := \|\nu\|_{L^{\infty}(P)}|P|\sum_{i=1}^{\#\texttt{dof}}\texttt{dof}_i\big((I-{\bm\Pi}_k^0)\vv\big)\texttt{dof}_i\big((I-{\bm\Pi}_k^0)\ww\big)\,,
$$
here $\#\texttt{dof}$ is the number of degrees of freedom associated with the polyhedron $P$ and
we define the function $\texttt{dof}_i$ that given a function in $\bm{\mathcal{V}}_h^k(P)$ returns its $i^{\text{th}}$ degrees of freedom.
Then, there exist two positive constants $\alpha_*,\alpha^*\in\R$ independent on the mesh size
such that the following relation holds
\begin{equation}
\alpha_*\,a^P(\ww,\ww) \leq a^P_h(\ww,\ww) \leq  \alpha^*a^P(\ww,\ww)\,.
\label{eqn:stability}
\end{equation}
\end{prop}
\begin{proof}
The result is a direct consequence of Proposition~\ref{prop:estUp} and~\ref{prop:estDown}.
\end{proof}

%% file: integration.tex
\section{Integration over curved Polyhedrons}\label{sec:integration}

In this section we explain how we compute integrals inside polyhedrons characterized by curved faces.
Fist, we describe how to integrate over curved faces.
Then, we show how we extend the idea proposed in~\cite{Sommariva2009} from curved polygons to polyhedrons with curved faces.

\subsection{Integration over curved 2d polygons}

We consider a curved face $F$ of a polyhedron and
we made the same assumpions on definition, existence and regularity on the map $\gamma$,
see the final part of Section~\ref{sec:notations}.

Given a function $f$ defined over the face $F$,
to integrate any function defined on such face,
we exploit the following formula:
\begin{eqnarray}
\int_{F} f(x,\,y,\,z)~\text{d}F = \int_{\mathfrak{F}} \tilde{f}(u,\,v)\left\|\frac{\partial \gamma}{\partial u}\times\frac{\partial \gamma}{\partial v}\right\|\ds
 \approx\sum_{i=1}^n \tilde{f}(u_i,\,v_i)\left\|\frac{\partial \gamma}{\partial u}\times\frac{\partial \gamma}{\partial v}(u_i,\,v_i)\right\|\omega_i\,,\nonumber \\
\label{eqn:intPara}
\end{eqnarray}
where $\tilde{f}$ is the function $f$ written in terms of the parameter-space coordinates $(u,\,v)$,
$\left\|\frac{\partial \gamma}{\partial u}\times\frac{\partial \gamma}{\partial v}\right\|$ is the Jacobian of the transformation $\gamma$
and $\left\{(u_i,\,v_i)\right\}_{i=1}^n$ are the quadrature points in the parameter space and $\{\omega_i\}_{i=1}^n$ are their corresponding weights.

From a more practical point of view the usage of \eqref{eqn:intPara} is not so straightforward.
Indeed, one has to define the function $\tilde{f}(u,\,v)$.
However, to use the quadrature rule in \eqref{eqn:intPara} and avoid the definition of the function $\tilde{f}$,
we compute the physical quadrature points corresponding to the one in the parameter space via the map~$\gamma$
$$
\left\{(u_i,\,v_i)\right\}_{i=1}^n\rightarrow\left\{(x_i,\,y_i,\,z_i)\right\}_{i=1}^n\,,
$$
and we modify \eqref{eqn:intPara} as
\begin{eqnarray}
\sum_{i=1}^n \tilde{f}(u_i,\,v_i)\left\|\frac{\partial \gamma}{\partial u}\times\frac{\partial \gamma}{\partial v}(u_i,\,v_i)\right\|\omega_i =
\sum_{i=1}^n f(x_i,\,y_i,\,z_i)\left\|\frac{\partial \gamma}{\partial u}\times\frac{\partial \gamma}{\partial v}(u_i,\,v_i)\right\|\omega_i\,.\nonumber \\
\label{eqn:faceTrick}
\end{eqnarray}
This quadrature rule does not depend on $\tilde{f}$ so there is no need to create such a function.
Furthermore such double quadrature point lists will play a key in the computation of the volume quadrature presented in Section~\ref{sec:3dInt}.

\begin{remark}
Since the domain ${\mathfrak{F}}\subset [0,\,1]^2$ in the parameter space may have curved boundaries,
we exploit the integration strategy proposed in~\cite{Sommariva2009,BeiraodaVeiga2019}.
\end{remark}

\subsection{Integration over curved 3d polyhedrons}\label{sec:3dInt}

In this section we propose a quadrature rule to integrate a function inside polihedrons characterized by curved faces.
The proposed formula is an extension of the ones proposed in~\cite{Sommariva2009,BeiraodaVeiga2019}.

Consider a function $f(x,\,y,\,z)$ defined in a polyhedron $P$ and suppose that we would like to compute the integral
$$
\int_P f(x,\,y,\,z)\dP\,.
$$
The idea is to move the computation of such integral
to an integration over polyhedron's faces.
To achieve this goal we exploit the divergence theorem.
We define a proper vectorial field $\F$ such that
$$
\diver\F = f\,.
$$
One possible choice of such function is the vector field
$$
\F(x,\,y,\,z) = \left[
\begin{array}{c}
0 \\
0 \\
\mathlarger{\int_{z_0}^z f(x,\,y,\,t)\dt}
\end{array}
\right]\,,
$$
where $z_0$ is the $z$-coordinate of the polyhedron's barycenter.
Staring from such vector field and exploiting the idea of \eqref{eqn:faceTrick},
we obtain the following integration formula
\begin{eqnarray}
 \int_P f\dP &=& \int_P \diver \F\dP =  \int_{\partial P} \nn\cdot\F\dF
 = \sum_{F\in\partial P} \int_F \nn\cdot\F\dF
 = \sum_{F\in\partial P} \int_F n_z\,F_z\dF\nonumber\\
 &=& \sum_{F\in\partial P} \int_{\mathfrak{F}} \tilde{n}_z(u,\,v)\,\tilde{F}_z(u,\,v)\left\|\frac{\partial \gamma}{\partial u}\times\frac{\partial \gamma}{\partial v}\right\|\ds\nonumber\\
 &\approx& \sum_{F\in\partial P} \sum_{i=1}^n \tilde{n}_z(u_i,\,v_i)\,\tilde{F}_z(u_i,\,v_i)\left\|\frac{\partial \gamma}{\partial u}\times\frac{\partial \gamma}{\partial v}(u_i,v_i)\right\|\omega_i\nonumber\\
 &=& \sum_{F\in\partial P} \sum_{i=1}^n n_z(x_i,\,y_i,\,z_i)\, F_z(x_i,\,y_i,\,z_i)\left\|\frac{\partial \gamma}{\partial u}\times\frac{\partial \gamma}{\partial v}(u_i,v_i)\right\|\omega_i\,.\nonumber\\
 \label{eqn:forWithoutFace}
\end{eqnarray}
In the last expression we are able to compute all terms.
Both $n_z$ and $\left\|\frac{\partial \gamma}{\partial u}\times\frac{\partial \gamma}{\partial v}\right\|$ are known functions
since they describe the $z$-component of the normal to the surface geometry and the Jacobian of the map $\gamma$, respectively.
Moreover, $F_z(x_i,\,y_i,\,z_i)$  is computable too.
Indeed, recalling the definition of the vector field $\F$,
we evaluate such term via an edge integral:
$$
F_z(x_i,\,y_i,\,z_i) = \int_{z_0}^{z_i} f(x_i,\,y_i,\,t)\dt \approx \sum_{j=1}^{m} f(x_i,\,y_i,\,z_j)\,\omega_j\,.
$$
Substiting the last equation in \eqref{eqn:forWithoutFace},
we get the following integration formula:
\begin{eqnarray}
\int_P f\dP = \sum_{F\in\partial P} \sum_{i=1}^n\sum_{j=1}^{m} n_z(x_i,\,y_i,\,z_i)\,f(x_i,\,y_i,\,z_j)\left\|\frac{\partial \gamma}{\partial u}\times\frac{\partial \gamma}{\partial v}(u_i,v_i)\right\|\omega_i\omega_j\,.\nonumber \\
\label{eqn:finalQPRule}
\end{eqnarray}
% This last equation can be re-written in terms of quadrature points and weights as
% \begin{equation*}
% \left\{(x_s,\,y_s,\,z_s)\right\}_{s=1}^{n\,m}\quad\text{and}\quad\left\{\sigma_s\right\}_{s=1}^{n\,m}
% \end{equation*}
% where we defined the following weights
% $$
% \sigma_s := \nn_z(x_i,\,y_i,\,z_j)\,\left\|\frac{\partial \gamma}{\partial u}\times\frac{\partial \gamma}{\partial v}(u_{\,\overline{i}},v_{\,\overline{i}})\right\|\omega_{\,\overline{i}}\,\omega_{\,\overline{j}}\,,
% $$
% for a particular choice of the indexes $i$ and $j$.

We can make the following observation about such formula.
First of all it relay on some reuglarity assumpions of polihedron's face:
they have to be defined via a map $\gamma$ that have to be at least $C^2$
since we exploit surface normals and the Jacobian.

The proposed quadrature rule has more than the expected quadrature points
that is not appealing from the computational point of view.
However, we can use the procedure proposed in~\cite{Sommariva2015}
to get a quadrature rule with the same order composed by a subset of the input points.

Such compression procedure is general and
it is based on the resolution of a non-negative least squares problem that also ensures positivity of the weights.
We apply such strategy to reduce the number of quadrature points in Equation~\eqref{eqn:finalQPRule}.

To show the effectiveness of the compression procedure,
we show the following example.
Consider a cube whose top face is a bilinear surface and a quadrature forumla that exactly integrate polynomials of degree 2.

In Figure~\ref{fig:qpCompress} we show the resulting quadrature points.
The savings in terms of opeartion is evident: we move from 192 to 10 quadrature points.
The main feature of this compression procedure is that
it {always} results in the minimum number of points that interpolate a specific polynomial~\cite{Sommariva2015}.
Indeed, in this case we are considering a quadrature rule of degree 2 so
the compression procedure select 10 quadrature points and properly modify their weights.

\begin{figure}[!htb]
\centering
\begin{tabular}{ccc}
\texttt{no compression} &\phantom{sssss} &\texttt{compression} \\
\includegraphics[width=0.4\textwidth]{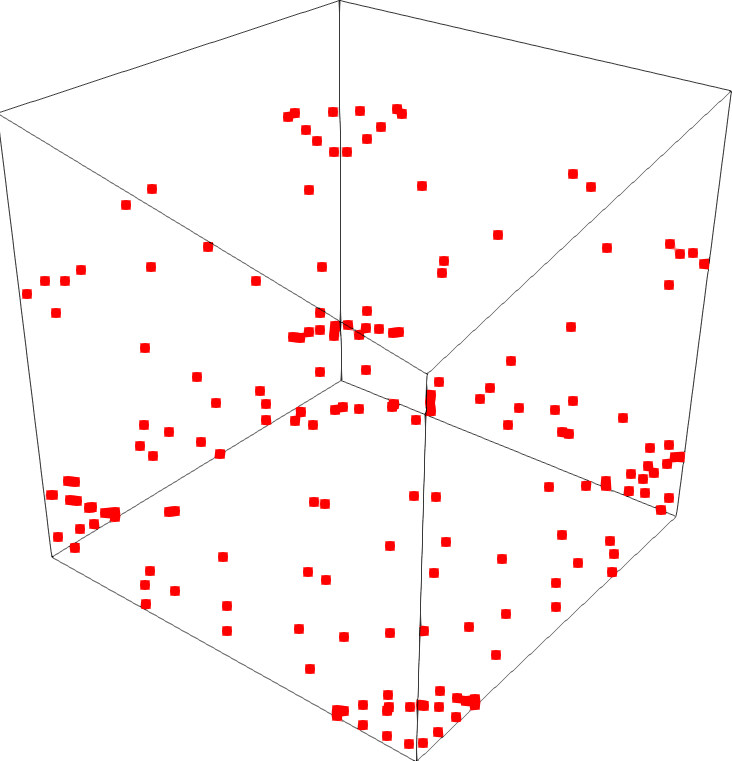}& &
\includegraphics[width=0.4\textwidth]{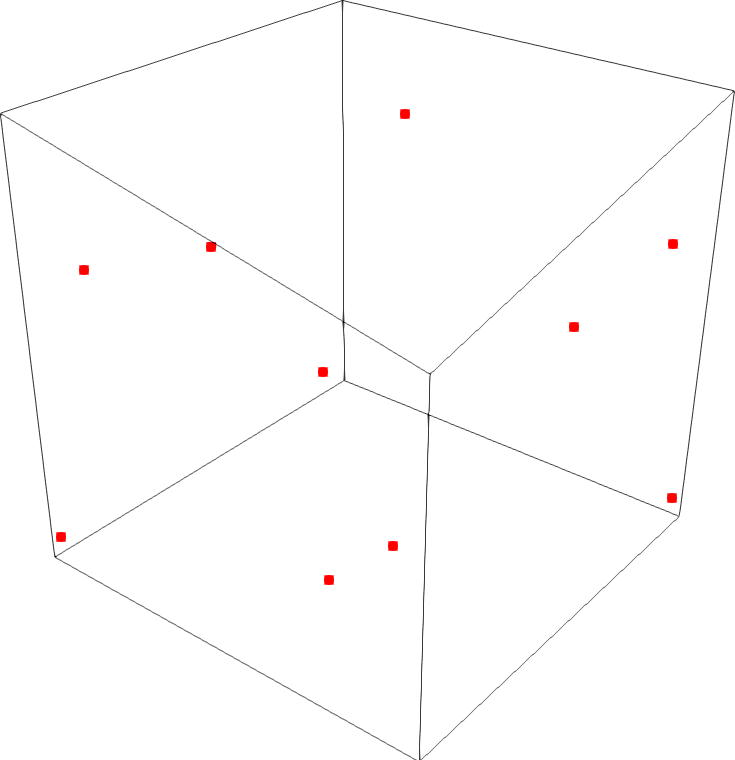}
\end{tabular}
\caption{Number of quadrature points without and with the compression procedure.
In the former case we have 192 points, while in the latter only 10 points.}
\label{fig:qpCompress}
\end{figure}

\begin{remark}
To compute $\F_z$ it is necessary the quote $z_0$.
A deeper analysis is required in finding such value, since the quadrature points may fall outside $P$.
This is an important issue since the function $f$ can be not defined outside $P$.
A good idea is to consider not a quote $z_0$ but a generic plane
that cuts the polyhedron in such a way that {all}
the quadrature points are inside $P$ or at least on the faces of the polyhedron.
\end{remark}

% \begin{figure}[!htb]
% \centering
% \begin{tabular}{cc}
% \includegraphics[width=0.45\textwidth]{imm/polyhedronQPZ.png} &
% \includegraphics[width=0.45\textwidth]{imm/polyhedronQPY.png}
% \end{tabular}
% \caption{Two different choices of the plane to define the quadrature rule of \eqref{eqn:quadRule}.
% On the left a plane $z=z_0$ where we highlight some quadrature points that fall outside the polyhedron,
% on the right a plane $y=y_0$, where we fixed the the poins which previosly fall outside the polyhedron.}
% \label{fig:cut}
% \end{figure}
%
% Finally the quadrature rule of \eqref{eqn:quadRule} we have more than the expected quadrature points.
% In~\cite{Sommariva2015} the authors present a compression procedure to reduce the number of nodes in the numerical quadrature.
% This procedure is general and
% it is based on the resolution of a Non Negative Least Squares problem that also ensures positivity of the weights.
% Consequently we adopt such general procedure for this case too.
%
\subsection{Preliminary numerical validation on quadrature weights}

Before dealing with the convergence analysis of Problem~\eqref{pb:weak},
we make a numerical example that focus on the proposed quadrature rule.
We consider the following domain
$$
\Omega :=\{(x,\,y,\,z)\in\R^3\::\: R_1^2\leq x^2+y^2\leq R_2^2,\, y\geq 0,\,0\leq z \leq 1\}\,,
$$
where $R_1$ and $R_2$ are 0.2 and 1., respectively,
and we explit the proposed quadrature formula, to compute the following integral
$$
\int_\Omega \sqrt{x^2+y^2} + z \dO\,.
$$
In Table~\ref{tab:resultVol} we compute the relative error
varying both the discretizations of $\Omega$ and the quadrature rule degree.
More specifically we are considering two meshes \texttt{cyli1} and \texttt{cyli2} shown in Figure~\ref{fig:meshToIntegrate}.

\begin{table}[!htb]
\centering
\begin{tabular}{ccc}
\multicolumn{3}{c}{Cylinder example}\\
\hline
Gau\ss{} degree &\texttt{cyli1} &\texttt{cyli2} \\
\hline
1 &1.1263e-03 &2.5196e-04 \\
2 &3.7429e-06 &2.6066e-06 \\
3 &6.0654e-08 &2.9604e-08 \\
\hline
\end{tabular}
\caption{}
\label{tab:resultVol}
\end{table}

\begin{figure}[!htb]
\centering
\begin{tabular}{cc}
\texttt{cyli1} &\texttt{cyli2} \\
\includegraphics[width=0.40\textwidth]{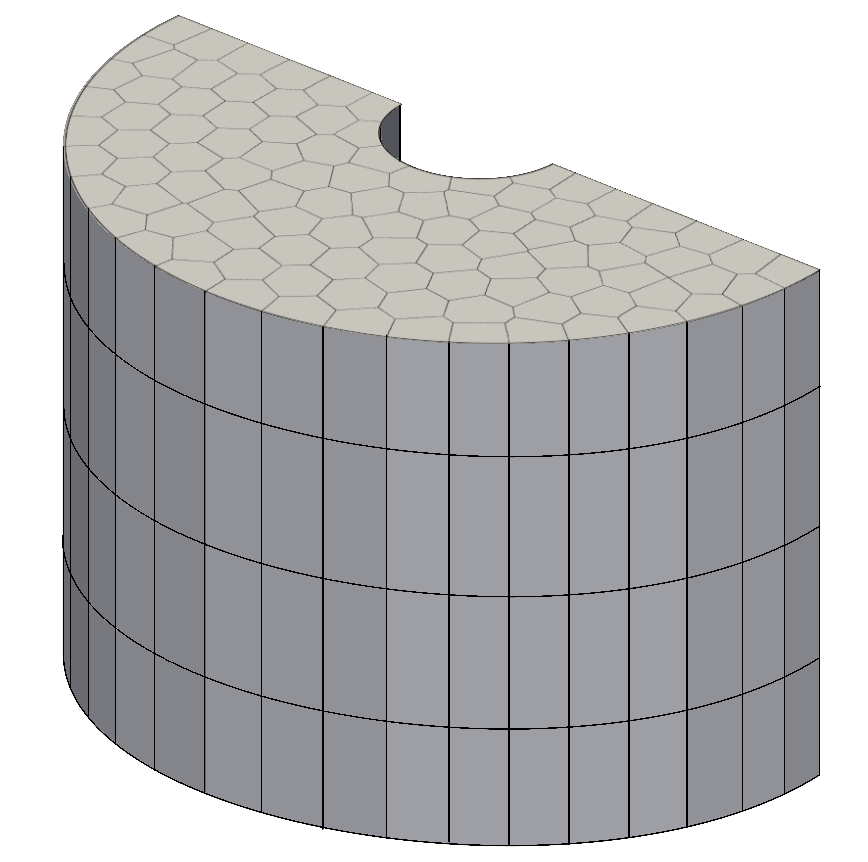} &
\includegraphics[width=0.40\textwidth]{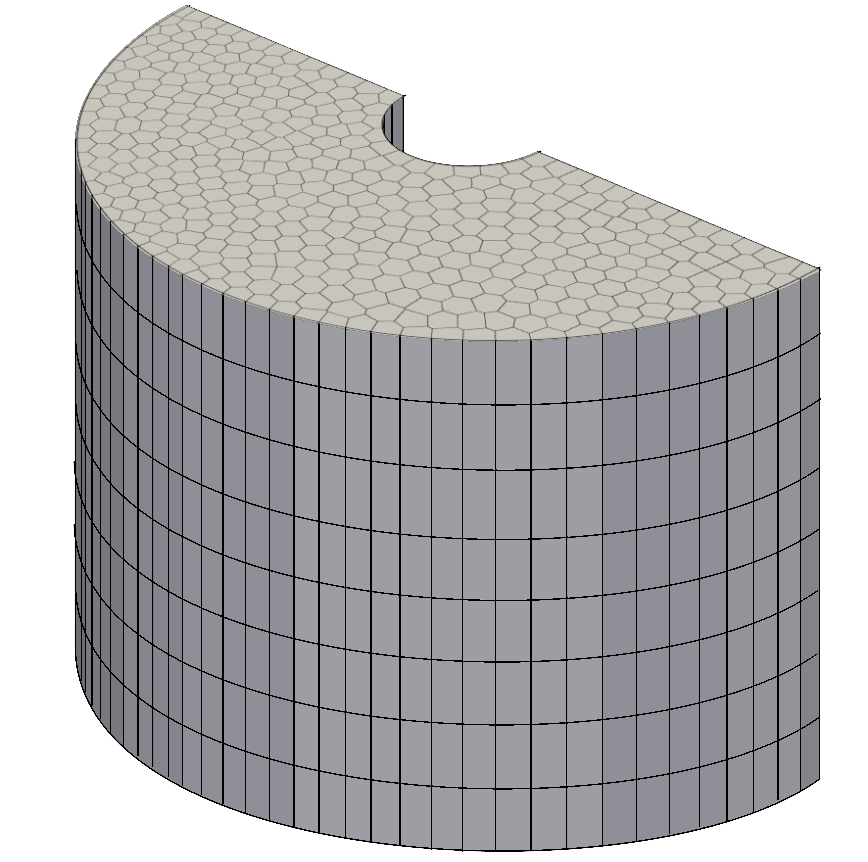} \\
\end{tabular}
\caption{Meshes used to produce the data in Table~\ref{tab:resultVol}.}
\label{fig:meshToIntegrate}
\end{figure}

\todo[inline]{cambia figura cyli2}

%% file: numExe.tex
\section{Numerical Examples}\label{sec:numExe}

In this section we give numerical evidence about the behaviour between straight
(\texttt{noGeo}) and curved (\texttt{withGeo}) approach.
We consider different kind of meshes and
describe how we build them inside each subsection.
Regardless of the type of tessellation
we associate to each mesh the mesh-size
$$
h = \frac{1}{N_P}\sum_{P\in\Omega_h} h_P\,,
$$
where $N_P$ is the number of polyhedrons in $\Omega_h$.

To proceed with the convergence analysis of the errors,
we build a sequence of meshes with decreasing $h$ and
we compute the following errors indicators:
\begin{itemize}
 \item $L^2$-error on the velocity field
 $$
 e_{L^2}^{\vv} := \sqrt{\sum_{P\in\Omega_h}\|\PiV\vv_h - \vv\|^2_{0}}\,,
 $$
 where $\PiV\vv_h$ is the element-wise $L^2$-projection operator of the virtual element variable $\vv_h$ inside the polyhedron $P$
 and $\vv$ is the known exact vector field;
 \item $L^2$-error on the pressure
 $$
 e_{L^2}^{p} := \sqrt{\sum_{P\in\Omega_h}\|p_h - p\|^2_{0}}\,,
 $$
 where $p_h$ is the element-wise polynomial that represents the pressure and
 $p$ is the exact pressure distribution.
\end{itemize}

From the theoretical point of view both error indicators have a convergence rate of $O\left(h^{k+1}\right)$.
However, when we consider $k=2,3$ and the \texttt{noGeo} case,
we get $O\left(h^2\right)$,
only the novel \texttt{withGeo} approach will have the expected trend for each~$k$.

Such behavior is justified by the following remark.
Let us suppose that the error in a numerical scheme can be split in two parts
\begin{equation}
\err = \errApp + \errGeo\,.
\label{eqn:genErr}
\end{equation}
On the one hand, we make an error in approximating the functional continuous spaces with a discrete one.
Such error is related to degree of the polynomial used inside the discrete space taken into account, $\errApp$.
On the other hand, we make an error in approximating the computational domain $\Omega$ with a mesh $\Omega_h$, $\errGeo$.

When we consider standard domains, i.e., domains whose boundaries are planes,
elements with straight faces perfectly match such boundaries so $\errGeo$ is null.
Consequently the error of \eqref{eqn:genErr} consists {only} on the first part and
we have
\begin{equation}
\err = \errApp + 0 = \errApp \sim O\left(h^{k+1}\right)\,,
\label{eqn:errVero}
\end{equation}
for either $e_{L^2}^{\vv}$ or $e_{L^2}^{p}$.
As a consequence we get the expected convergence rates since the geometrical error disappears.

However, when we are dealing with domains whose boundaries are curved,
if we approximate them via elements with straight faces,
the geometrical error is not null.
It is related {only} to the mesh size
and not on the approximation of the functional spaces we are using.
More specifically, if we consider straight faces,
it is {always} $O\left(h^2\right)$.
Consequently, for PDEs defined over curve domains \eqref{eqn:genErr} becomes
\begin{equation*}
\err = \errApp + \errGeo \sim O\left(h^{k+1}\right) + O\left(h^{2}\right)  = O\left(h^{\min\{k+1,2\}}\right)\,,
%\label{eqn:realErr}
\end{equation*}
Thus
when we have a curved boundaries and we consider a mesh with straight faces,
the expected rate is achieved {only} if $k=1$,
otherwise $\errGeo$ overcomes $\errApp$.

If we use the proposed approach, i.e.,
we modify the functional spaces in such a way that they can handle elements with curved faces,
the geometrical error is null.
Indeed, since we put inside the space the map $\gamma$ that {exactly} describes the geometry,
we do not make any error in approximating curved boundaries or interfaces via the functional spaces.
Thus, we recover the expected convergence rate as in \eqref{eqn:errVero}.

\paragraph{Meshes}
In the following examples we will use meshes whose boundary faces are quadrilaterals.
We make this choice to allow the comparison between the \texttt{noGeo} and \texttt{withGeo} case.
Indeed, if we consider a polygon with more than four edges,
it is not a priori guaranteed that its vertexes lay on the same plane so
we need to sub-triangulate such polygons in order to proceed with the \texttt{noGeo} approach.
Then, if we make such sub-triangulation,
the \texttt{noGeo} and \texttt{withGeo} case are not comparable in terms of degrees of freedom
since the former requires more degrees of freedom with respect to the latter.

\subsection{Example 1: convergence analysis with Dirichlet boundary}

We consider a domain composed by 5 planar faces and one curvilinear face defined by
$$
\Gamma = \left\{(x,\,y,\,z)\in\R^3\::\: z +\frac{1}{10}\,\sin(\pi x)-1 = 0\right\}\,.
$$
We define the right-hand side and the boundary conditions in such a way
that the solution of Problem~\eqref{pb:weak} is
$$
p(x,\,y,\,z) := \left(z+\frac{1}{10}\,\sin(\pi x)-1\right)^2\quad\text{and}\quad \bm{q} = -\nabla p\,,
$$
with $\kappa=\mu=1$ and we impose essential boundary conditions on all the boundaries including the curved one.
In Figure~\ref{fig:meshExe1} we show one of the mesh taken into account
to collect the error for the convergence analysis of the errors $e_{L^2}^{\vv}$ and~$e_{L^2}^{p}$.

\begin{figure}[!htb]
\centering
\includegraphics[width=0.45\textwidth]{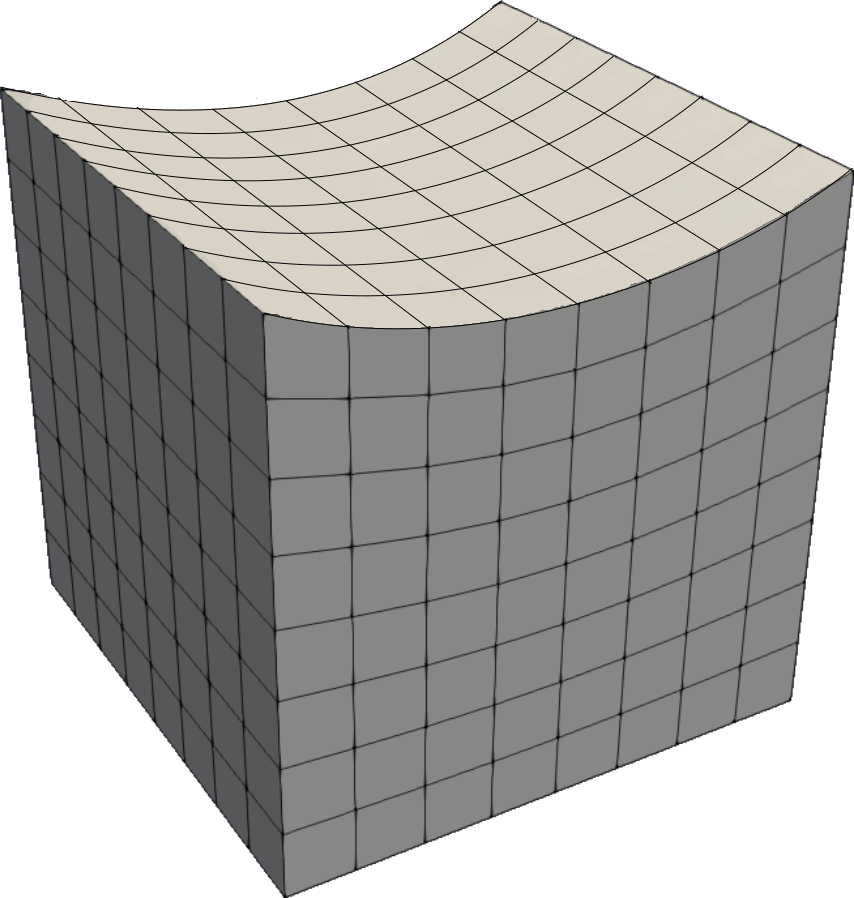}
\caption{Example 1: one of the mesh taken into account for the convergence
analysis for the \texttt{noGeo} case.}
\label{fig:meshExe1}
\end{figure}

In Figure~\ref{fig:convExe1} we show the convergence lines for such errors varying the VEM approximation degrees for both the \texttt{noGeo} and \texttt{withGeo} case.
More specifically, dashed and full lines represent the error obtained via the \texttt{noGeo} and \texttt{withGeo} approach.

As it was already discussed at the beginning of this section,
for $k=1$ these two numerical schemes behave as expected, i.e.,
both have an error trend of $O(h^2)$.
However, if we consider $k=2$ and $3$ their trends are different.
Indeed, the \texttt{noGeo} behaves as $O(h^2)$, while
the \texttt{withGeo} one as $O(h^{k+1})$.

\begin{figure}[!htb]
\centering
%\begin{tabular}{cc}
%\includegraphics[width=0.48\textwidth]{imm/exe1UErr-crop.pdf} &
%\includegraphics[width=0.48\textwidth]{imm/exe1PErr-crop.pdf}
%\end{tabular}
\includegraphics[width=1\textwidth]{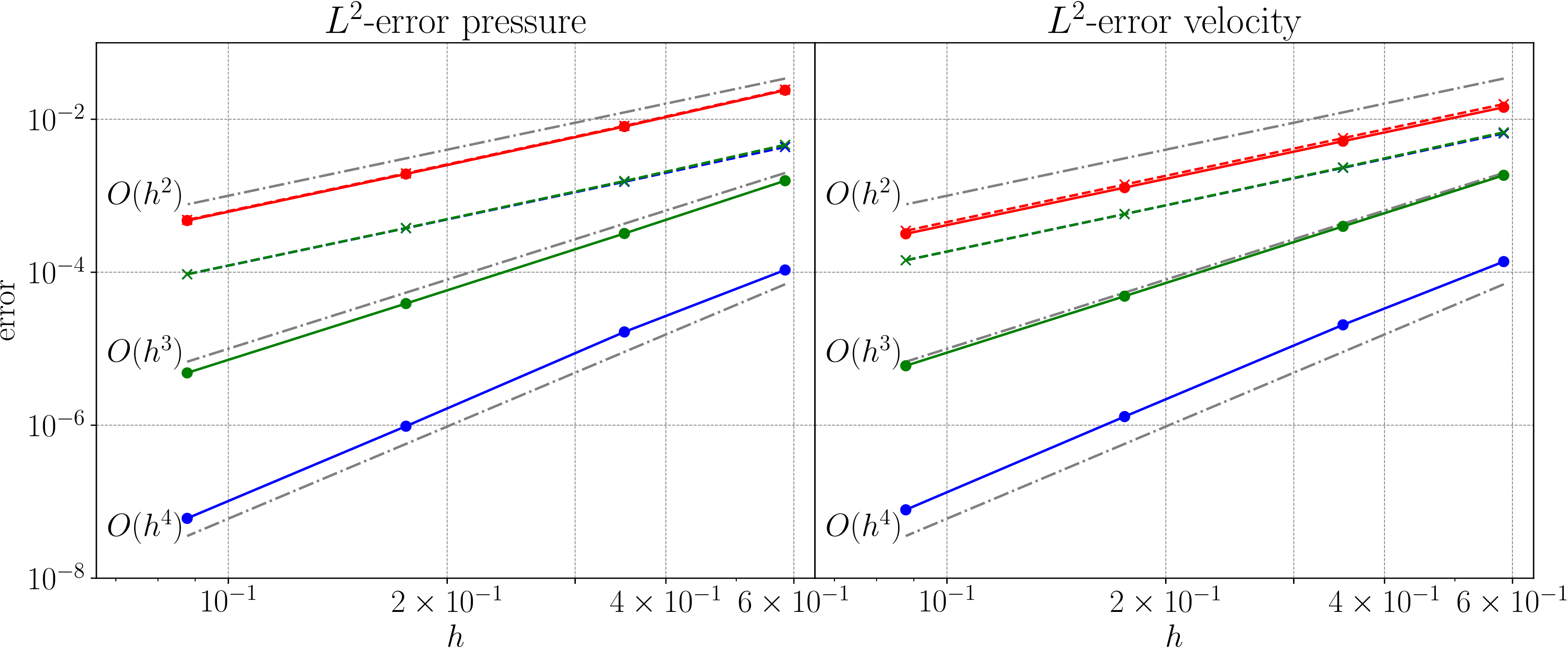}\\[0.25cm]
\includegraphics[width=0.75\textwidth]{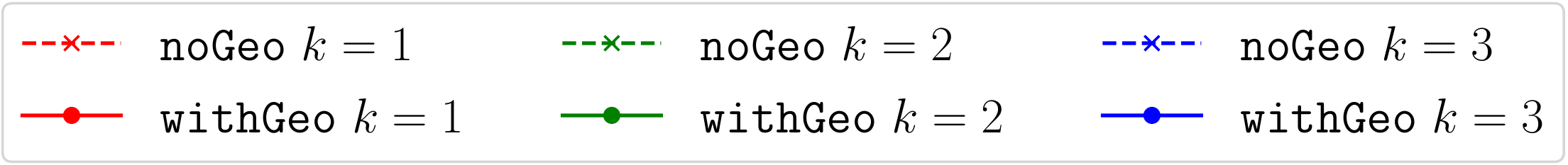}
\caption{Example 1: convergence lines for both $e_{L^2}^{\vv}$ and $e_{L^2}^{p}$.}
\label{fig:convExe1}
\end{figure}

By comparing such convergence lines we have the numerical evidence about what we
inferred at the beginning of this section.
Also the lines of the error for the $\texttt{noGeo}$ case coincides varying the degree $k>1$, i.e.,
\errGeo{} is too large and such convergence lines are actually $\errGeo$ for each degree $k$.

\subsection{Example 2: convergence analysis with Neumann boundary}

In this section we consider as computational domain
$$
\Omega :=\{(x,\,y,\,z)\in\R^3\::\: R_1^2\leq x^2+y^2\leq R_2^2,\, y\geq 0,\,0\leq z \leq 1\}\,,
$$
where $R_1$ and $R_2$ are 0.2 and 1., respectively.
On such domain we consider Problem~\ref{pb:weak} where the exact solution is
$$
p(x,\,y,\,z) = \sin(\pi x)\,\cos(\pi y)\,\sin(\pi z)\quad\text{and}\quad\bm{q} = -\nabla p\,.
$$
and on planar boundaries we impose essential boundary conditions,
while on the inner and outer curved boundary we impose natural boundary conditions.

To get a computational domain of such geometry,
we extrude along the $z$ axis a two-dimensional polygonal mesh, see Figure ~\ref{fig:meshExe2}.
We consider two types of two dimensional meshes to extrude:
a mesh composed by squares and
one composed by triangles.
We refer to such meshes as \texttt{quad} and \texttt{tria}, respectively.

\begin{figure}[!htb]
\centering
\begin{tabular}{cc}
\multicolumn{2}{c}{\texttt{tria}}\\
\includegraphics[width=0.4\textwidth]{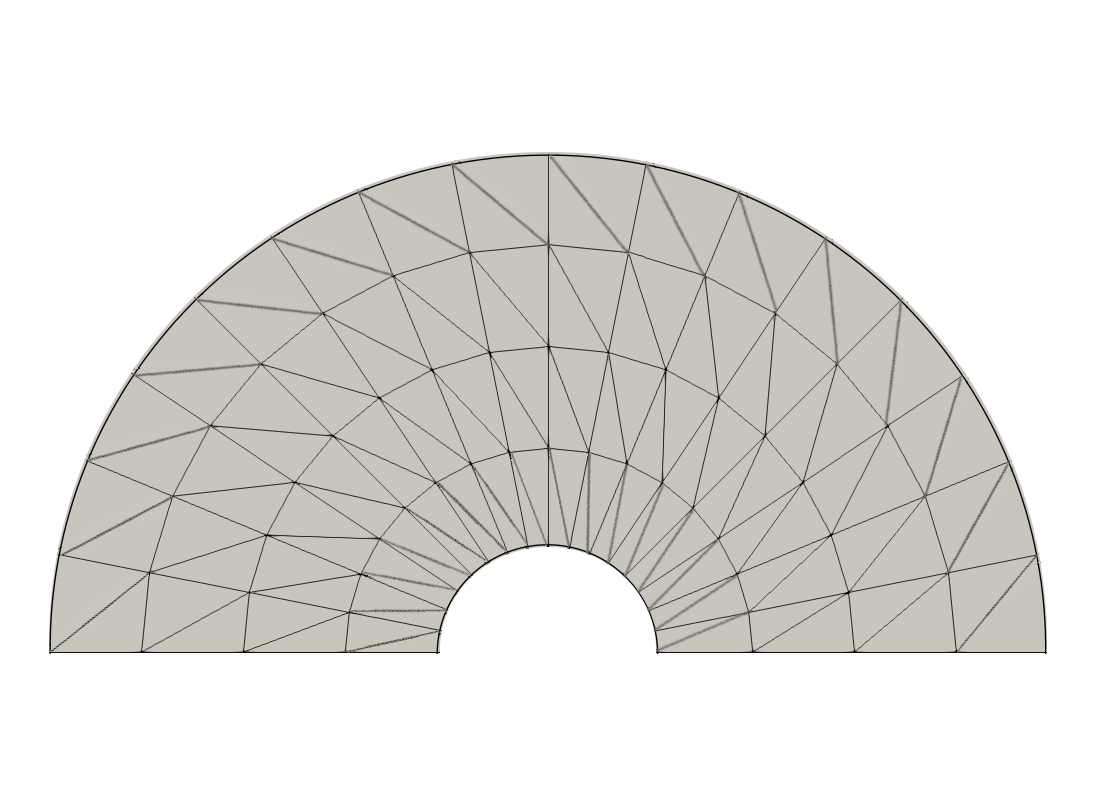} &\includegraphics[width=0.4\textwidth]{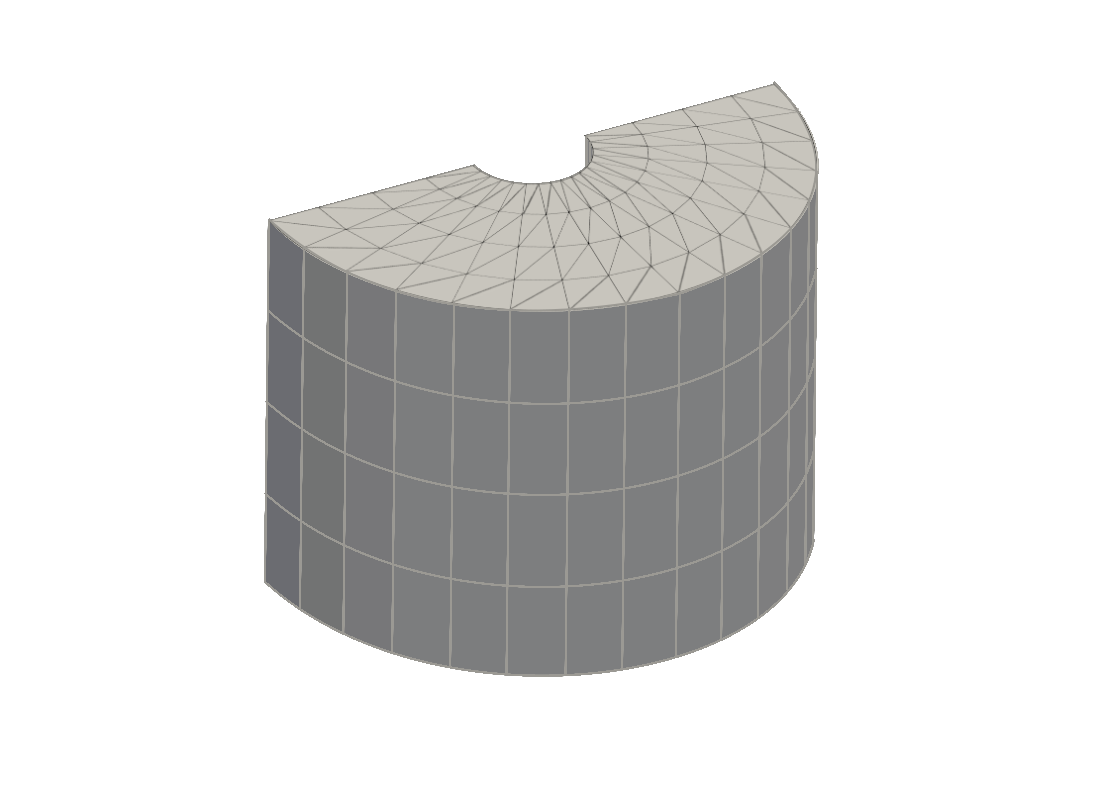} \\
\multicolumn{2}{c}{\texttt{quad}}\\
\includegraphics[width=0.4\textwidth]{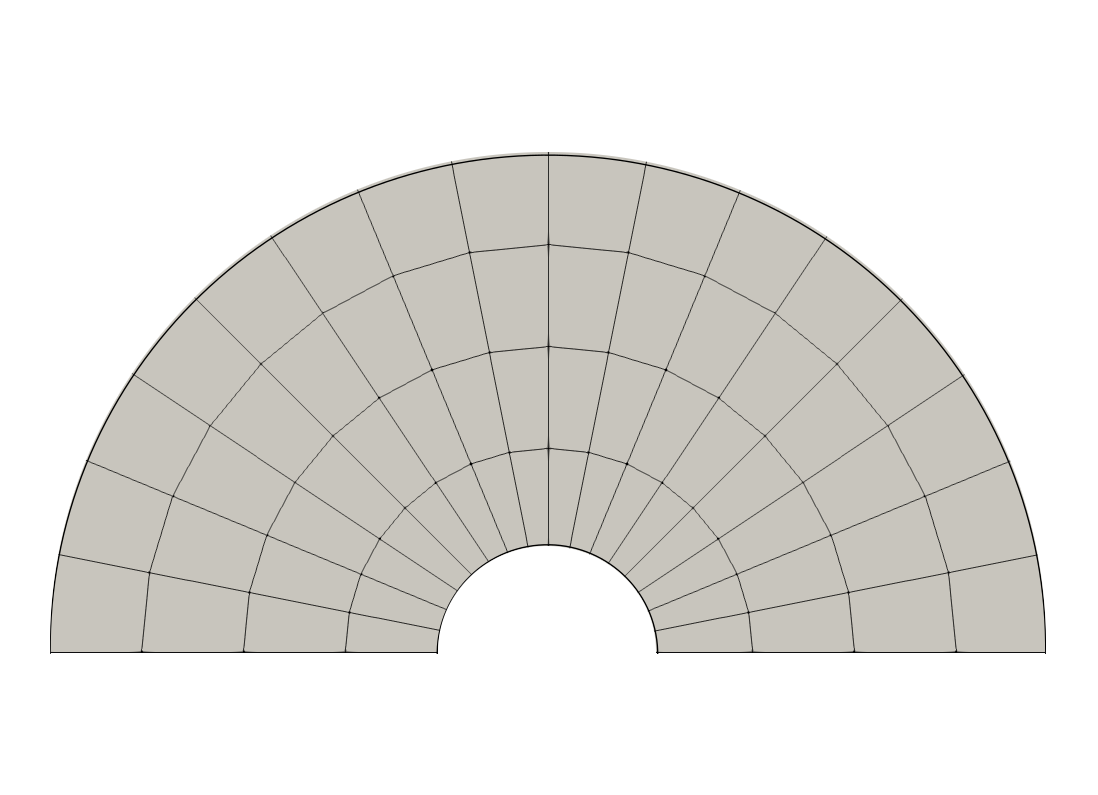} &\includegraphics[width=0.4\textwidth]{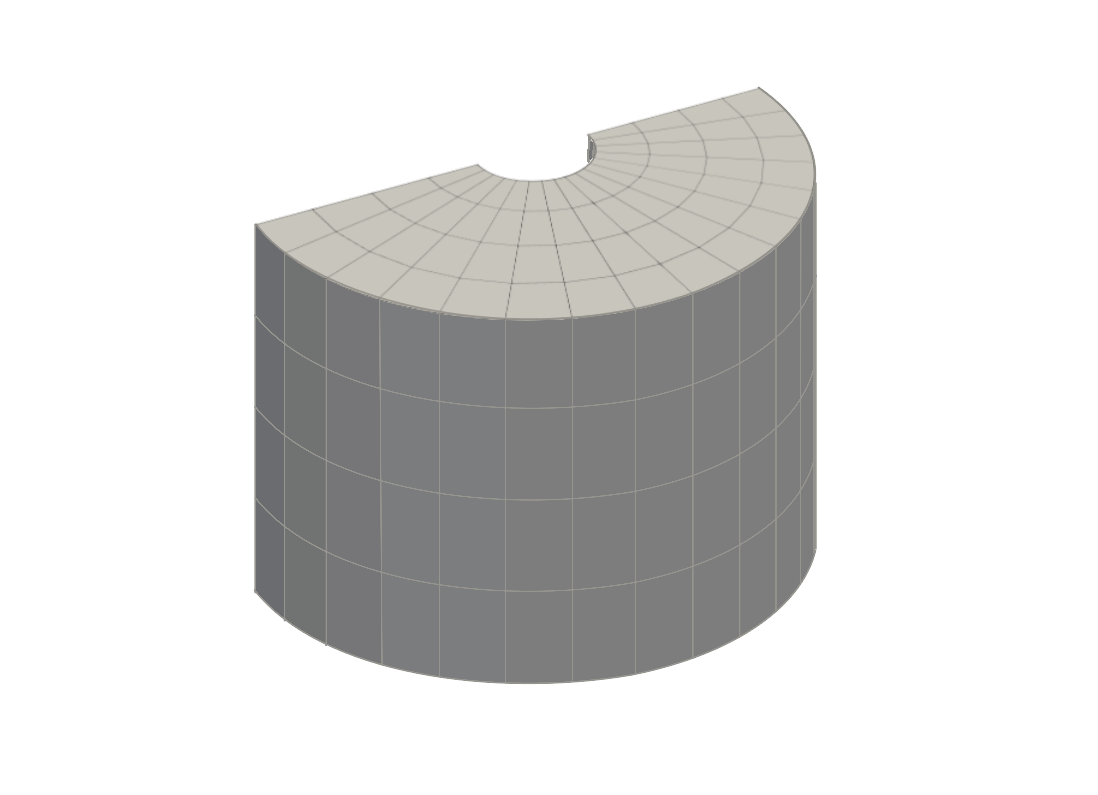}
\end{tabular}
\caption{Example 2: on the left the two-dimensional meshes we are extruding, on
the right the final meshes. In both cases we are showing the \texttt{noGeo}
meshes.}
\label{fig:meshExe2}
\end{figure}

In Figure~\ref{fig:convExe2} we collect the convergence lines for each type of mesh.
As already seen in the previous example we got the expected error trend for degrees $k=2$ and 3 {only} with \texttt{withGeo} approach,
while the \texttt{noGeo} one has always a trend of $O(h^2)$ for both $e_{L^2}^{\vv}$ and $e_{L^2}^{p}$ and for each approximation degree.
Moreover, also in this case the dotted associated with $k>1$ coincides.
Such fact gives a further numerical evidence that \errGeo{} overcomes \errApp{} and
what is actually represented by the dashed lines in Figure~\ref{fig:convExe2} is \errGeo{}.

\begin{figure}[!htb]
\centering
\includegraphics[width=1\textwidth]{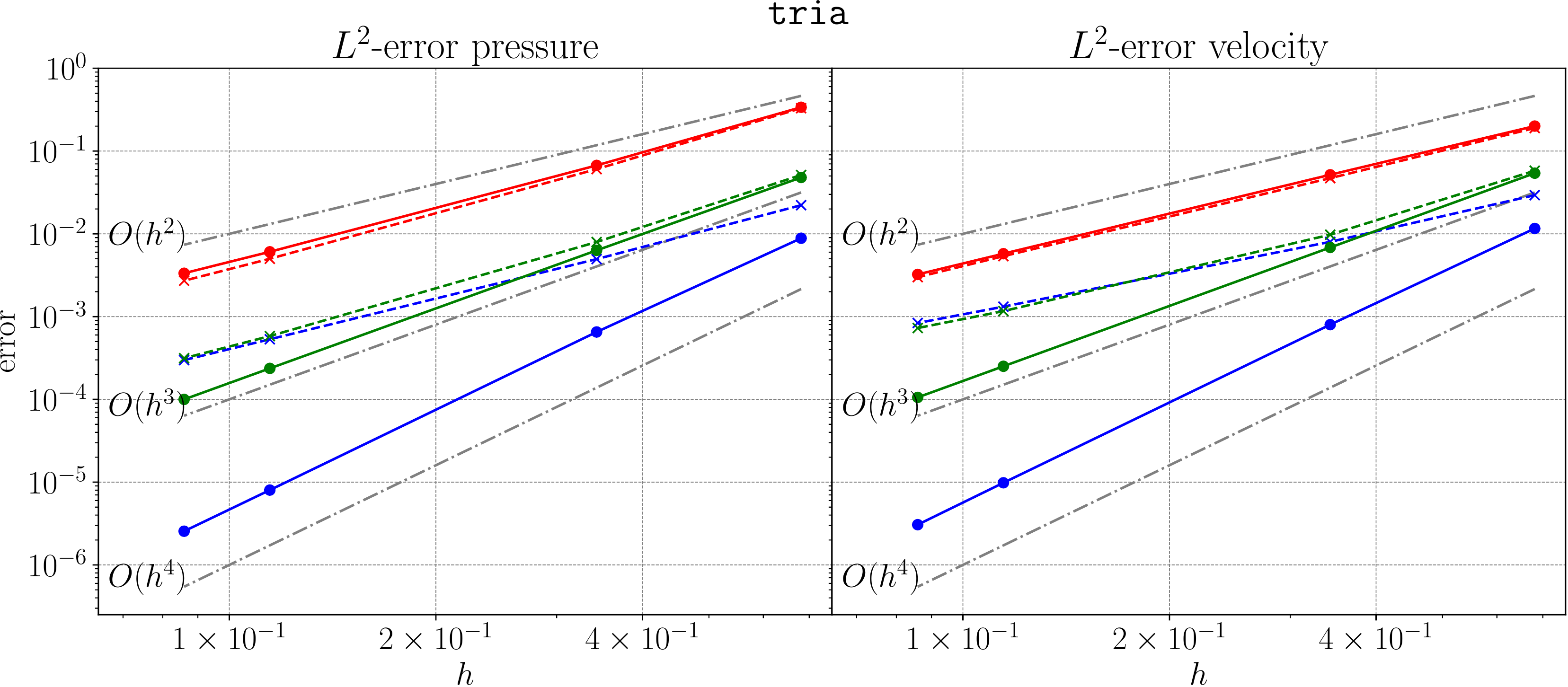}\\
\includegraphics[width=1\textwidth]{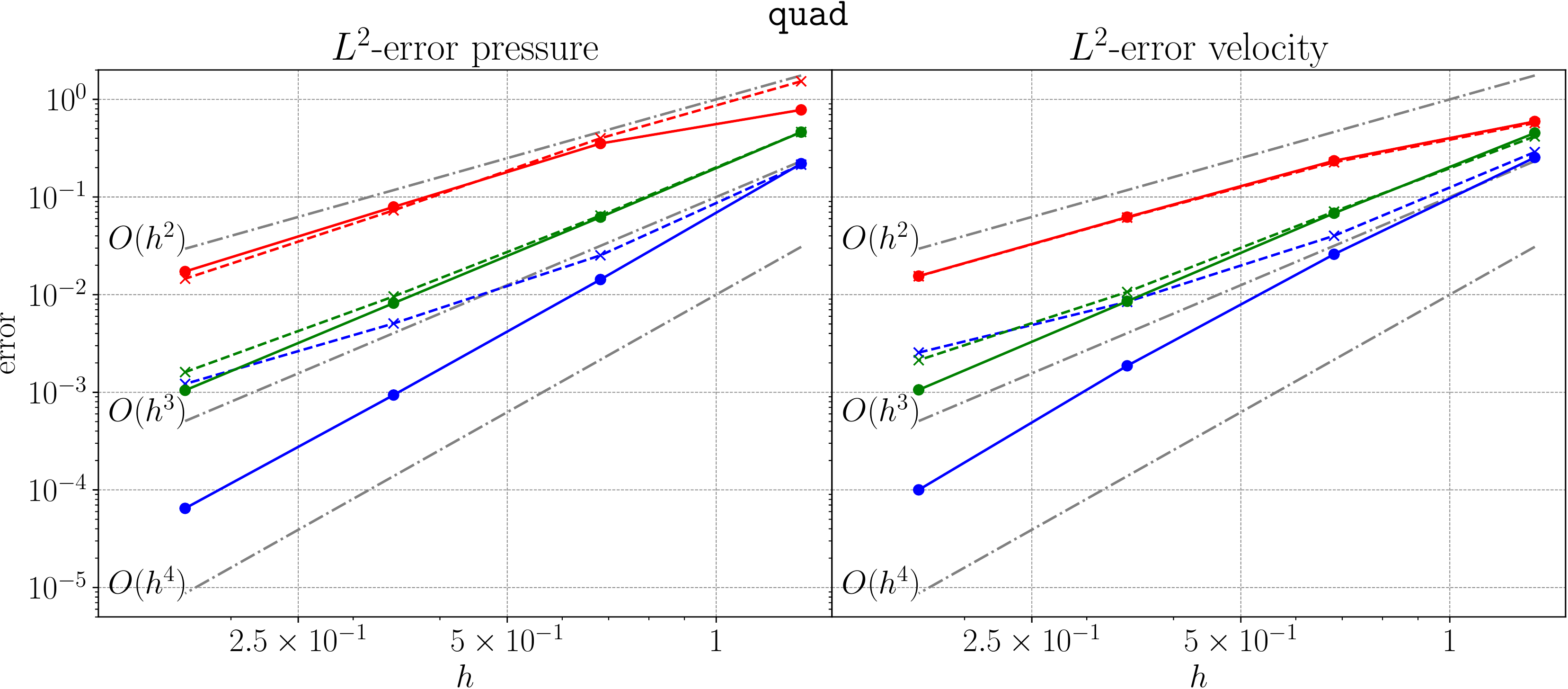}\\[0.25cm]
\includegraphics[width=0.75\textwidth]{imm/legend}
%\begin{tabular}{cc}
%\multicolumn{2}{c}{\texttt{tria}}\\
%\includegraphics[width=0.48\textwidth]{imm/exe2UErrTria-crop.pdf} &
%\includegraphics[width=0.48\textwidth]{imm/exe2PErrTria-crop.pdf}\\[0.5em]
%\multicolumn{2}{c}{\texttt{quad}}\\
%\includegraphics[width=0.48\textwidth]{imm/exe2UErrQuad-crop.pdf} &
%\includegraphics[width=0.48\textwidth]{imm/exe2PErrQuad-crop.pdf}\\[0.5em]
%\end{tabular}
\caption{Example 2: convergence lines for both $e_{L^2}^{\vv}$ and $e_{L^2}^{p}$.}
\label{fig:convExe2}
\end{figure}

\subsection{Example 3: corner point meshes}

We consider now a particular kind of curved cells: the corner point elements.
Such cells have the topology of a cube but their top and bottom faces are bilinear surfaces, i.e.,
they are defined by the following map $\gamma:[0,\,1]^2\to\mathbb{R}^3$
$$
\gamma(u,\,v) = (1-u)(1-v)\xx_A+u(1-v)\xx_B+ uv\,\xx_C+u\,(1-v)\xx_D\,,
$$
where $\xx_A,\,\xx_B,\,\xx_C$, and $\xx_D$ are the 4 points at the top (or at the bottom) of the cell.
In oil industry basins and reservoirs are usually described by this type of meshes~\cite{Aarnes2008}.

We solve Problem~\eqref{pb:weak} in a unit cube $\Omega=[0,\,1]^3$
composed of three layers of materials characterized by different values of permeability.
More specifically, the top and the bottom layers have $\kappa=1.$ while the middle one has $\kappa=0.01$.
We set natural boundary conditions on the top and bottom faces of the domain, i.e.,
the pressure variable is equal to zero at the top and equals to one at the bottom.
Then, we consider the force term equal to zero.
We consider only a VEM approximation degree $k=2$ as a representative degree,
similar consideration can be done for $k>2$.

Since we do not have the exact solution, we can not compute the error and
we will give only a qualitative analysis on the pressure we get.
Moreover we consider three refinement levels of the mesh at hand to ensure
that the method does converge to plausible solution, we refer to them as \texttt{ref 1}, \texttt{2} and \texttt{3}.

\begin{figure}[!htb]
\centering
\begin{tabular}{ccc}
\texttt{ref 1} &
\texttt{ref 2} &
\texttt{ref 3} \\[1em]
\includegraphics[width=0.3\textwidth]{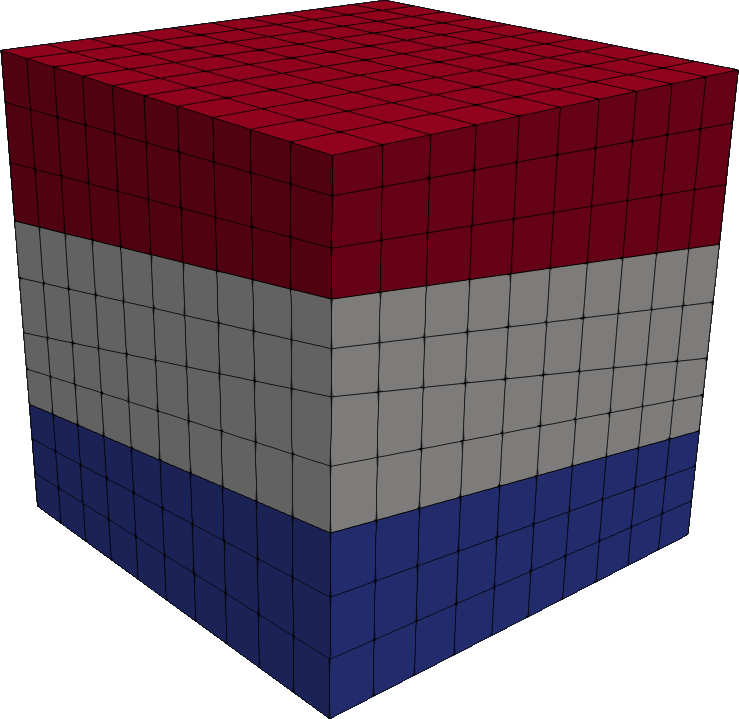} &
\includegraphics[width=0.3\textwidth]{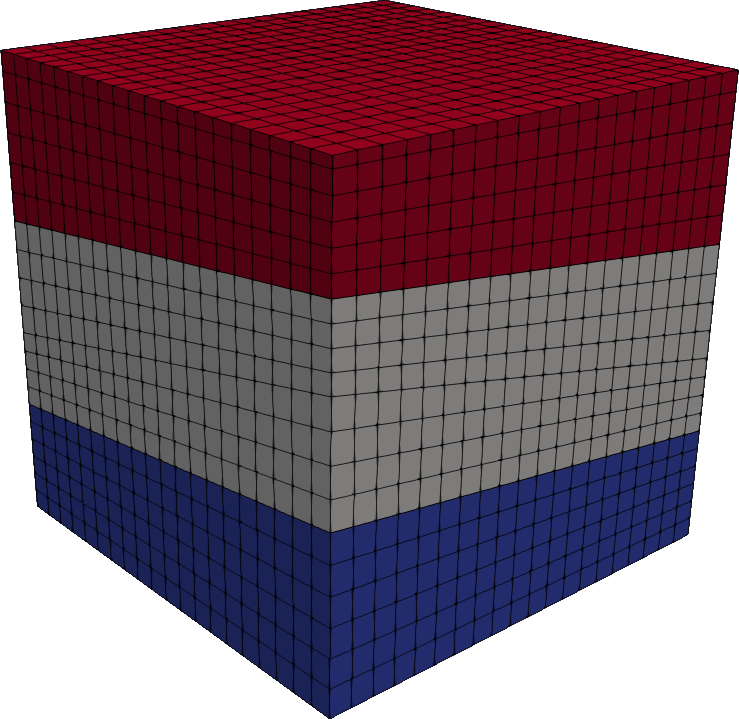} &
\includegraphics[width=0.3\textwidth]{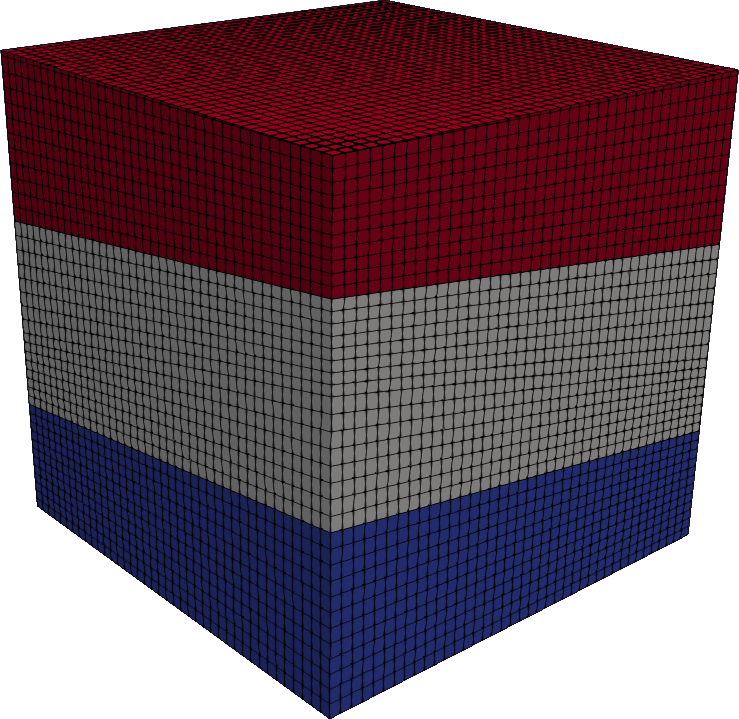} \\
\end{tabular}
\caption{Example 3: refinement levels of the mesh with different layers.}
\label{fig:pressField}
\end{figure}

We further underline that in such discretizations all the mesh elements are corner point cells
so the proposed virtual element incorporates the curved geometry within the function space definition.
Thanks to this approach we do not need to introduce any approximation of the curved faces.
Moreover we do not need to sub-triangulate them and introduce further degrees of freedom and,
consequently, increase the size of the linear system at hand.

\begin{figure}[!htb]
\centering
\begin{tabular}{ccc}
\texttt{ref 1} &
\texttt{ref 2} &
\texttt{ref 3} \\[1em]
\includegraphics[width=0.3\textwidth]{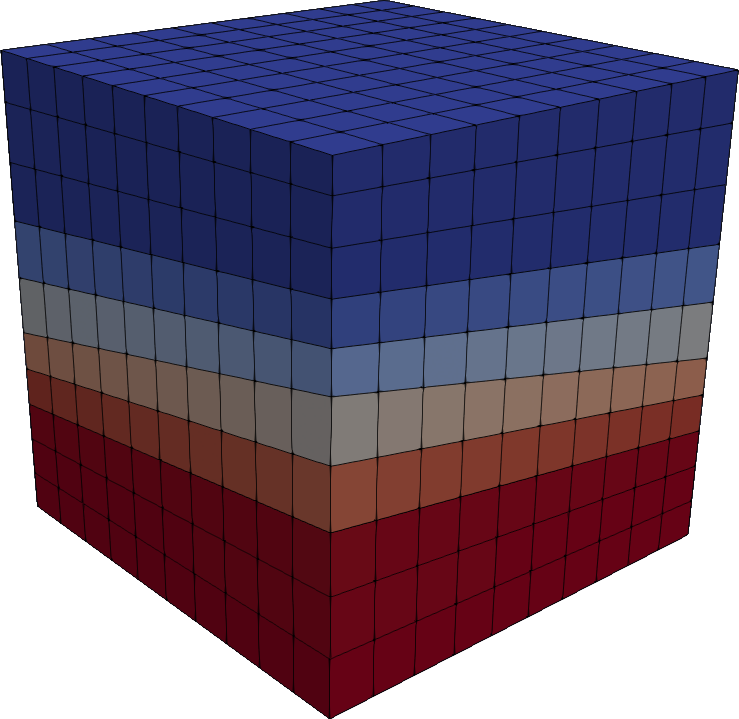} &
\includegraphics[width=0.3\textwidth]{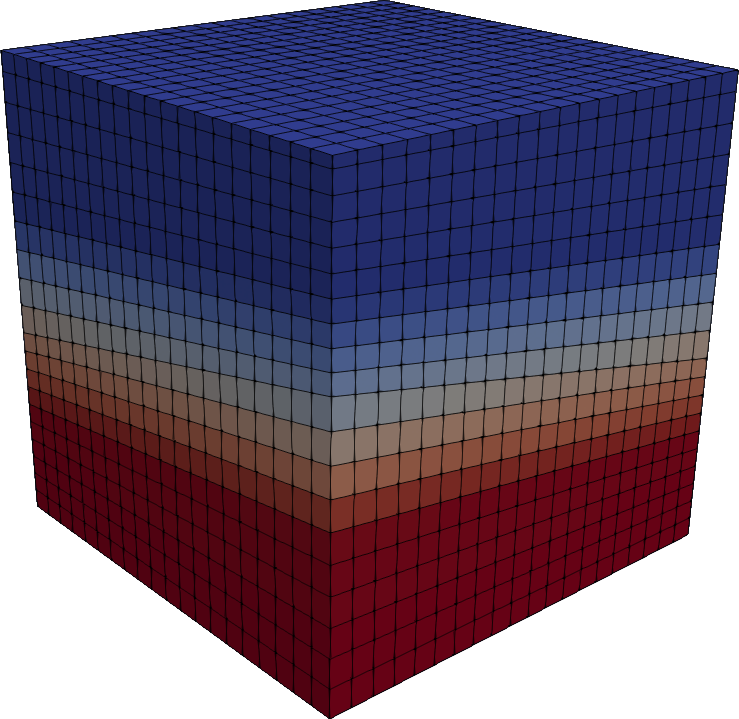} &
\includegraphics[width=0.3\textwidth]{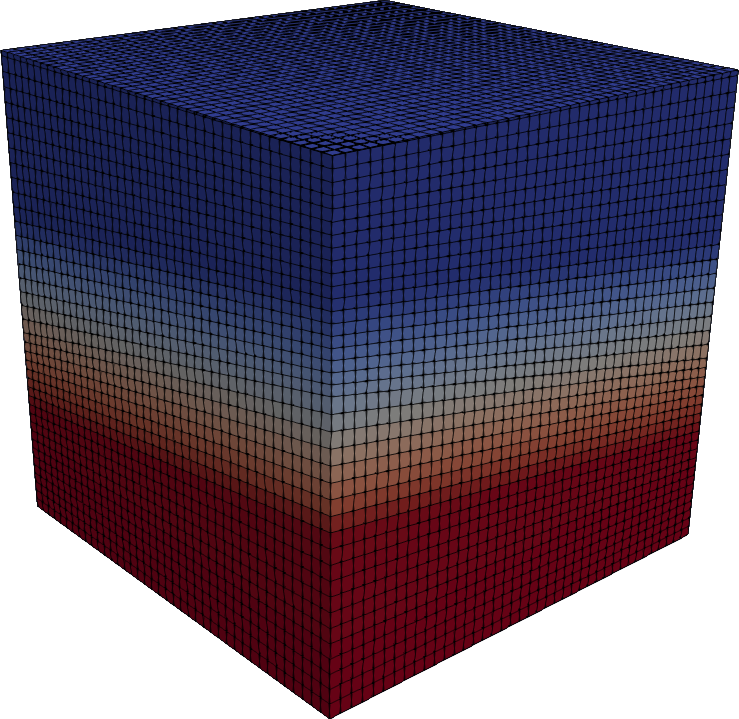} \\[1em]
\multicolumn{3}{c}{\includegraphics[width=0.50\textwidth]{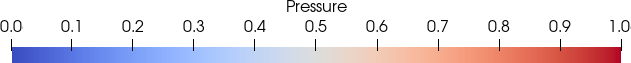}}
\end{tabular}
\caption{Example 3: pressure filed with three different refinement levels.}
\label{fig:pressField1}
\end{figure}

From the data in Figure~\ref{fig:pressField} we observe a steep change of the pressure values.
This fact is expected since we are considering different values of the permeability $\kappa$.
As it was expected, such change of pressure is better captured by mesh refinement.

%% file: conc.tex
\section{Conclusions}\label{sec:conc}

In this work, we have proposed a mixed virtual element scheme in three space dimensions to solve Darcy problems. 
The considered scheme handles curved portions of boundary or internal interfaces without any degradation of the expected order of convergence rate. 
Moreover, it can be seen as a natural extension of the standard case,
indeed the proposed functional spaces coincide with the ones used for the standard case if the element has no curved faces.
We also proposed a theoretical analysis 
to show how to define the approximation spaces, degrees of freedom, and stabilization form and to obtain a stable numerical scheme. 
Such method also required the definition quadrature rule able to deal with polyhedrons characterized by curved faces.
In this paper we propose a possible strategy which still need a deep theoretical analysis,
but it is validated by ad-hoc numerical experments and 
implicitly proven by the solution of Darcy problems.
Finally, several numerical tests further show that 
the proposed scheme achieves the expected order without any degradation due to geometrical errors.

\section*{Acknowledgments}
This paper is the last checkpoint of the project ``Bend VEM 3d''.
All authors acknowledge INdAM-GNCS which makes possible this interesting and long journey.